\pdfoutput=1
\documentclass[a4paper]{amsart}
\usepackage{amssymb}
\title{$P$-adic Asai $L$-functions of Bianchi modular forms}
\date{}
\author{David Loeffler}
\address{Mathematics Institute, University of Warwick, Coventry CV4 7AL, United Kingdom}
\email{d.a.loeffler@warwick.ac.uk}
\urladdr{http://orcid.org/0000-0001-9069-1877}
\thanks{Supported by Royal Society University Research Fellowship ``$L$-functions and Iwasawa theory'' (Loeffler).}

\author{Chris Williams}
\address{Imperial College London, South Kensington Campus, London SW7 2AZ, United Kingdom}
\curraddr{Mathematics Institute, University of Warwick, Coventry CV4 7AL, United Kingdom}
\email{Christopher.D.Williams@warwick.ac.uk}
\urladdr{http://orcid.org/0000-0001-8545-0286}

\usepackage[margin=1.32in]{geometry}

\usepackage{preamble}
\usepackage{enumerate, mathtools}

\numberwithin{equation}{section}

\begin{document}

\begin{abstract}
The Asai (or twisted tensor) $L$-function of a Bianchi modular form $\Psi$ is the $L$-function attached to the tensor induction to $\Q$ of its associated Galois representation. In this paper, when $\Psi$ is ordinary at $p$ we construct a $p$-adic analogue of this $L$-function: that is, a $p$-adic measure on $\Zp^\times$ that interpolates the critical values of the Asai $L$-function twisted by Dirichlet characters of $p$-power conductor. The construction uses techniques analogous to those used by Lei, Zerbes and the first author in order to construct an Euler system attached to the Asai representation of a quadratic Hilbert modular form.
\end{abstract}


\maketitle

%
%

\section{Introduction}

\subsection{Background}

 Several of the most important conjectures in modern number theory, such as the Bloch--Kato and Beilinson conjectures, relate the special values of $L$-functions to arithmetic data. In much of the work on these conjectures to date, an important role has been played by \emph{$p$-adic $L$-functions}: measures or distributions on $\Zp^\times$, for a prime $p$, interpolating the special values of a given complex $L$-function and its twists by Dirichlet characters of $p$-power conductor. Such functions are expected to exist in wide generality, but in practice they can be difficult to construct, and there are large classes of $L$-functions which at present are not known to have a $p$-adic analogue. In this paper, we provide such a construction for a new class of $L$-functions: the \emph{Asai}, or \emph{twisted tensor}, $L$-functions attached to Bianchi modular forms (automorphic forms for $\GLt / F$, where $F$ is imaginary quadratic).

 In order to construct our $p$-adic $L$-function, we use the Betti cohomology of a locally symmetric space associated to $\GLt/F$. Work of Ghate \cite{Gha99} shows that the critical values of the Bianchi Asai $L$-function and its twists are computed by certain special elements in Betti cohomology, which can be reinterpreted as pushforwards of cohomology classes for $\GLt/\Q$ associated to Eisenstein series. However, interpolating such classes $p$-adically is not straightforward. The key novelty in our construction is to \emph{simultaneously} vary two parameters: the choice of Eisenstein series, and the choice of embedding of $\GLt/\Q$ in $\GLt/F$. This allows us to reduce the interpolation problem to a (much simpler) compatibility property of the $\GLt/\Q$ Eisenstein series.

 Our construction uses techniques that are closely related to those those found in \cite{LLZ14} and \cite{LLZ16}, in which Lei, Zerbes and the first author constructed Euler systems (certain compatible families of \'etale cohomology classes) for Rankin--Selberg convolutions of modular forms, and for the Asai representation of a Hilbert modular form over a real quadratic field. In the Bianchi setting, there is no \'etale cohomology to consider, since Bianchi manifolds (the symmetric spaces associated to $\GLt / F$) are not algebraic varieties. However, we show in this article that applying the same techniques in this setting instead gives compatible families of classes in the Betti cohomology of these spaces. Hence the same techniques used to construct an Euler system for $\GLt / F$ when $F$ is real quadratic also give rise to a $p$-adic $L$-function when $F$ is imaginary quadratic.

 We hope that these techniques can be extended to build other new $p$-adic $L$-functions as ``Betti counterparts'' of known Euler system constructions; in particular, we are presently exploring applications of this method to the standard $L$-function of (possibly non-self-dual) cohomological automorphic representations of $\operatorname{GL}_3 / \Q$.

 \subsubsection*{Note} While working on this project, we learned that Balasubramanyam, Ghate and Vangala have also been working on a construction of $p$-adic Asai $L$-functions for Bianchi cusp forms \cite{BGV}. Their work is independent of ours, although both constructions rely on the same prior work \cite{Gha99} of Ghate.

\subsection{Statement of the main theorem}

 Let $\Psi$ be a Bianchi modular form of weight $(k, k)$, for some $k \ge 0$, which is an eigenform for the Hecke operators. We assume that the level $\n$ of $\Psi$ is divisible by all primes $\mathfrak{p} \mid p$ of $F$; this leads to no loss of generality, since we may replace $\Psi$ by a $\mathfrak{p}$-stabilisation if necessary. We let $N$ be the integer generating the ideal $\n \cap \Z$.

 The Asai $L$-function of $\Psi$ is defined by
 \[
  L^{\mathrm{As}}(\Psi, s) \defeq L^{(N)}(\varepsilon_{\Psi, \Q}, 2s-2k-2) \sum_{n\ge 1}c(n\roi_F, \Psi) n^{-s},
 \]
 where $\varepsilon_{\Psi, \Q}$ is the restriction to $\widehat{\Z}^\times$ of the nebentypus character of $\Psi$, $L^{(N)}(-)$ denotes the Dirichlet $L$-function with its Euler factors at primes dividing $N$ omitted, and $c(\m, \Psi)$ is the Hecke eigenvalue of $\Psi$ at the ideal $\m$.  We define similarly the twisted Asai $L$-function $L^{\mathrm{As}}(\Psi, \chi, s)$, for a Dirichlet character $\chi$ (see Definition \ref{def:asaiLS} below).

 This function $L^{\mathrm{As}}(\Psi, s)$ has an Euler product, in which the factors at primes $\ell \nmid N$ can be interpreted in terms of Galois representations: they are the local Euler factors of the tensor induction to $\operatorname{Gal}(\overline{\Q}/\Q)$ of the compatible family of 2-dimensional $\ell$-adic representations of $\operatorname{Gal}(\overline{\Q}/F)$ attached to $\Psi$. (We shall discuss the Euler factors at the ``bad'' primes in \S \ref{sect:primitiveL} below.) The Asai $L$-function should not be confused with the standard $L$-function $L^{\mathrm{std}}(\Psi, s) \defeq \sum_{\m \trianglelefteqslant \roi_F}c(\m, \Psi) \operatorname{Nm}(\m)^{-s}$, which corresponds to the usual (rather than tensor) induction of the same family of $\operatorname{Gal}(\overline{F}/F)$-representations.

 We choose a finite extension $L/\Qp$, containing $F$ and the Hecke eigenvalues of $\Psi$, with ring of integers $R$, and we assume that $\Psi$ is \emph{ordinary} at $p$, i.e.~$c(p\roi_F, \Psi)$ is a unit in $R$.

 \begin{theorem}
  For any integer $c > 1$ coprime to $6\n$, there exists a $p$-adic measure
  \[
  \subc L_p^{\mathrm{As}}(\Psi) \in R[[\Zp^\times]]
  \]
  on $\Zp^\times$ satisfying the following interpolation property: if $\chi$ is a Dirichlet character of conductor $p^r$, and $0 \le j \le k$, then we have
  \[
  \int_{\Zp^\times}\chi(x)x^j\, \mathrm{d}\subc L_p^{\mathrm{As}}(\Psi)(x) = \left\{\begin{array}{ll}(*)L^{\mathrm{As}}(\Psi,\chibar,j+1) &: \chi(-1)(-1)^j = 1,\\
  0 &: \chi(-1)(-1)^j = -1,
  \end{array}\right.
  \]
  where $(*)$ is an explicit factor (which is always non-zero if $r \ge 1$).
 \end{theorem}

See Theorem \ref{thm:interpolation} of the main text for a precise statement. Note that $\subc L_p^{\mathrm{As}}$ interpolates all critical values in the left half of the critical strip. The critical values in the right half are related to these via a functional equation, but we do not make this explicit.

It is possible to remove the dependence on $c$ entirely if the restriction to $\Q$ of the character of $\Psi$ is non-trivial and does not have $p$-power conductor. If this condition is not satisfied, then we can only remove $c$ at the cost of passing to a slightly larger space of ``pseudo-measures'', which may be interpreted as meromorphic (rather than analytic) functions on $p$-adic weight space; this is a $p$-adic counterpart of the fact that, for certain eigenforms $\Psi$, the Asai $L$-function and its twists can have poles. The details of this are contained in \S\ref{sec:getting rid of c}.

 \subsection{Outline of the construction} We first give a brief outline of the construction in the simplest case, when $\Psi$ is a normalised Bianchi modular eigenform of weight $0$ (i.e.~contributing to cohomology with trivial coefficients) for some imaginary quadratic field $F$.
 From $\Psi$ we construct a class $\phi_\Psi^* \in \h^1_{\mathrm{c}}(Y_{F, 1}^*(\n), R)$, where $Y_{F, 1}^*(\n)$ is a Bianchi manifold with appropriate level structure. This cohomology group is a free $R$-module of finite rank, and its $\Zp$-linear dual is $\h^2(Y_{F, 1}^*(\n), R) / (\text{torsion})$. In \cite{Gha99}, Ghate showed that critical values of the Asai $L$-function can be obtained by pairing $\phi_\Psi^*$ with certain classes in this $\h^2$ coming from classical weight 2 Eisenstein series. The main new ideas in the present paper arise in controlling integrality of these Eisenstein classes as the level varies, thus putting them into a compatible family from which we build a $p$-adic measure.

 The first input in our construction is a collection of maps, one for each $m \ge 1$ and $a \in \roi_F$, defined by
 \[
  Y_{\Q, 1}(m^2 N) \labelrightarrow{\iota} Y_{F, 1}^*(m^2 \n) \labelrightarrow{\kappa_{a/m}} Y_{F, 1}^*(\n),
 \]
 where $\iota$ is the natural embedding, and $\kappa_{a/m}$ is obtained by twisting the natural quotient map by $\smallmatrd{1}{a/m}{0}{1}$. Here $Y_{\Q, 1}(m^2 N)$ is the usual (open) modular curve for $\GLt / \Q$ of level $m^2 N$, where $N = \n \cap \Z$.

 The second input is a collection of special cohomology classes (``Betti Eisenstein classes'') ${}_c C_{m^2N} \in \h^1( Y_{\Q, 1}(m^2N), \Z)$. These are constructed using Siegel units. The theory of Siegel units shows that these classes satisfy norm-compatibility properties as $m$ varies, and that their images in de Rham cohomology are related to the Eisenstein series used in \cite{Gha99}. (The factor $c$ refers to an auxiliary choice of integer which serves to kill off denominators from these classes).

 With these definitions, we set
 \begin{align*}
  {}_c \Xi_{m, \n, a} \defeq (\kappa_{a/m} \circ \iota)_*\left( {}_c C_{m^2 N} \right) &\in \h^2(Y_{F, 1}^*(\n), \Z),\\
  \subc \Phi_{\n, a}^r \defeq \sum_{t \in (\Z / p^r \Z)} {}_c \Xi_{p^r, \n, at} \otimes [t] &\in \h^2(Y_{F, 1}^*(\n), \Z) \otimes \Zp[(\Z / p^r)^\times].
 \end{align*}
 The key theorem in our construction (Theorem \ref{norm relation}) is that the classes $\subc \Phi_{\n, a}^r$ satisfy a norm-compatibility relation in $r$. Both the statement of this norm-compatibility relation, and its proof, are very closely analogous to the norm-compatibility relations for Euler system classes in \cite{LLZ14, LLZ16}.

 From this, it follows that after renormalising using the Hecke operator $(U_p)_*$ (the transpose of the usual $U_p$) the classes $\subc\Phi_{\n, a}^r$ form an inverse system. In particular, they fit together to define an element
 \[
  \subc\Phi_{\n,a}^\infty \in e_{\mathrm{ord}, *}\h^2(Y_{F,1}^*(\n),\Zp)\otimes_{\Zp} \Zp[[\Zp^\times]],
 \]
where $e_{\mathrm{ord}, *}$ is Hida's ordinary projector asociated to $(U_p)_*$. We view this as a bounded measure on $\Zp^\times$ with values in the $(U_p)_*$-ordinary part of $\h^2(Y_{F, 1}^*(\n), \Zp)$. We then define the $p$-adic Asai $L$-function to be the measure
\[
 \subc L_p^{\mathrm{As}}(\Psi) \defeq \langle \phi_{\Psi}^*, \subc\Phi_{\n,a}^\infty\rangle \in R[[\Zp^\times]].
\]
That this measure interpolates the critical values of the (complex) Asai $L$-function then follows from \cite{Gha99} together with certain twisting maps (to obtain twisted $L$-values).

The case of higher-weight Bianchi forms (contributing to cohomology with non-constant coefficients) is similar, although unavoidably a little more technical. Suppose $\Psi$ is such a form of weight $(k,k)$. Using \cite{Gha99} and the same twisting methods as in the weight $(0,0)$ case, one can prove algebraicity for the critical value $L^{\mathrm{As}}(\Psi,\chi,j+1)$, where $0 \le j \le k$ and $\chi(-1)(-1)^j = 1$, by pairing with classes in $\h^2$ arising from Eisenstein series of weight $2k-2j+2$. For each such $j$, we define a compatible system of cohomology classes with coefficients in a suitable algebraic representation of $\GLt / F$ by applying a $p$-adic moment map to our Siegel-unit classes, obtaining classes $\subc\Phi_{\n,a}^{\infty,j}$ analogous to $\subc\Phi_{\n,a}^\infty$ in the weight $(0,0)$ construction. Again, this is a ``Betti analogue'' of a construction for \'etale cohomology which is familiar in the theory of Euler systems \cite{kings15,KLZ17}.

Pairing $\phi_\Psi^*$ with $\subc\Phi_{\n,a}^{\infty,j}$ gives a $p$-adic measure on $\Zp^\times$, as above. Using Kings' theory of $p$-adic interpolation of polylogarithms, it turns out that after a twist by the norm this measure is actually independent of $j$, and we define the $p$-adic Asai $L$-function $\subc L_p^{\mathrm{As}}(\Psi)$ to be the measure for $j=0$. Moreover, the class $\subc\Phi_{\n,a}^{\infty,j}$ can be explicitly related to weight $2k-2j+2$ Eisenstein series, so that integrating the function $\chi(x)x^j$ against $\subc L_p^{\mathrm{As}}(\Psi)$ computes the value $L^{\mathrm{As}}(\Psi,\chi,j+1)$ (under the parity condition above).

\subsection{Acknowledgements}

 The authors would like to thank Aurel Page for suggesting the proof of Proposition \ref{prop:goodinclusion}; and the two anonymous referees, who provided valuable comments and corrections on an earlier draft of the paper.

%
%
\section{Preliminaries and notation}\label{preliminaries}
\subsection{Basic notation}

We fix notation for a general number field $K$, which will either be $\Q$ or an imaginary quadratic field. (We'll generally denote this imaginary quadratic field by $F$ to distinguish it from the rationals in the notation). Denote the ring of integers by $\roi_K$, the adele ring by $\A_K$ and the finite adeles by $\A_K^f$. We let $\roikhat \defeq \widehat{\Z}\otimes_{\Z}\roi_K$ be the finite integral adeles, and $K^{\times +}$ the totally-positive elements of $K^\times$ (so that $K^{\times +} = K^\times$ for $K = F$).

Let $\uhp \defeq \{z \in \C: \mathrm{Im}(z)>0\}$ be the usual upper half-plane, with $\GLt(\R)_+$ (the group of $2 \times 2$ matrices of positive determinant) acting by M\"obius transformations in the usual way; we extend this to all of $\GLt(\R)$ by letting $\smallmatrd{-1}{}{}1$ act via $x + iy \mapsto -x + iy$.

Define the \emph{upper half-space} to be
\[ \uhs \defeq \{(z,t)\in \C\times\R_{>0}\}, \]
with $\GLt(\C)$ acting via
\[ \smallmatrd{a}{b}{c}{d} \cdot (z, t) = \left( \frac{(az + b)\overline{(cz + d)} + a\bar{c} t^2}{|cz + d|^2 + |c|^2 t^2}, \frac{ |ad-bc|t}{|cz + d|^2 + |c|^2 t^2}\right).\]
We embed $\uhp$ in $\uhs$ via $x + iy \mapsto (x, y)$, which is compatible with the actions of $\GLt(\R)$ on both sides.

Throughout, $p$ will denote a rational prime. Let $F$ be an imaginary quadratic field of discriminant $-D$, with different $\mathcal{D} = (\sqrt{-D})$, and fix a choice of $\sqrt{-D}$ in $\C$. Let $\n \subset \roi_F$ be an ideal of $F$, divisible by all the primes of $F$ above $p$; this will be the level of our Bianchi modular form. We assume throughout that $\n$ is small enough to ensure that the relevant locally symmetric space attached to $\n$ is smooth (see Proposition \ref{prop:goodinclusion}). Let $N$ be the natural number with $(N) = \Z\cap \n$ as ideals in $\Z$ (noting that $p \mid N$).

For an integer $n \ge 0$ and a ring $R$, define $V_n^{(r)}(R)$ to be the space of homogeneous polynomials of degree $n$ in two variables $X, Y$ with coefficients in $R$, with $\GLt(R)$ acting on the right via $(f \mid \gamma)(X, Y) = f(aX + bY, cX + dY)$. Similarly, we write $V_n^{(\ell)}(R)$ for the same space with $\GLt$ acting on the left, so $(\gamma \cdot f)(X, Y) = f(aX + cY, bX + dY)$.


\subsection{Locally symmetric spaces}

\begin{definition}\label{def:G*}
 Let $G$ be the algebraic group $\mathrm{Res}_{F/\Q}\GLt$ over $\Q$, and let $G^*$ be the subgroup
 $G \times_{D}\mathbb{G}_m$, where $D \defeq\mathrm{Res}_{F/\Q}\mathbb{G}_m$ and the map $G \rightarrow D$ is determinant.
\end{definition}

(Compare \cite[Definition 2.1.1]{LLZ16} in the totally-real case.)

\begin{definition}
 We define locally symmetric spaces attached to the groups $\GLt$, $G$ and $G^*$ as follows:
\begin{itemize}
 \item If $U \subset \GLt(\A_\Q^f)$ is an open compact subgroup, we set
 \[ Y_\Q(U) \defeq \GLt(\Q)_+\backslash \left[\GLt(\A_\Q^f)\times \uhp \right]/U,\]
 where $\GLt(\Q)_+$ acts from the left on both factors in the usual way, and $U$ acts on the right of $\GLt(\A_\Q^f)$.

 \item If $U \subset G(\A_\Q^f) = \GLt(\A_F^f)$ is open compact, we set
 \[ Y_F(U) \defeq \GLt(F)\backslash \left[\GLt(\A_F^f)\times \uhs \right]/U.\]

 \item If $U \subset G^*(\A_\Q^f) = \left\{ g \in \GLt(\A_F^f) : \det(g) \in (\A_\Q^f)^\times\right\}$ is open compact, we set
 \[ Y_F^*(U) \defeq G^*(F)_+ \backslash \left[G^*(\A_\Q^f) \times \uhs\right] / U, \]
 where $G^*(F)_+ = \{ g \in G^*(F) : \det(g) > 0\}$ is the intersection of $G^*(F)$ with the identity component of $G^*(\R)$.

\end{itemize}
\end{definition}

Each of these spaces has finitely many connected components, each of which is the quotient of $\uhp$ or $\uhs$ by a discrete subgroup of $\operatorname{PSL}_2(\R)$ (resp.~$\operatorname{PGL}_2(\C)$). If $U$ is sufficiently small, these discrete subgroups act freely, so in particular the quotient is a manifold.

\begin{definition}
 Let $K$ be either $\Q$ or $F$, and $\m$, $\n$, $\aaa$ be ideals in $\roi_K$. We define:
\begin{itemize}
\item[(i)] $U_K(\m,\n) \defeq \{\gamma \in \GLt(\roikhat): \gamma \equiv I \newmod{\smallmatrd{\m}{\m}{\n}{\n}}\},$
\item[(ii)] $U_K(\m(\aaa),\n) \defeq \{\gamma \in \GLt(\roikhat): \gamma \equiv I \newmod{\smallmatrd{\m}{\m\aaa}{\n}{\n}}\},$
\end{itemize}
We write $Y_K(\m, \n) \defeq Y_K(U(\m, \n))$ and similarly $Y_K(\m(\aaa), \n)$. We will be particularly interested $Y_K(\m, \n)$ for $\m = (1)$, which we abbreviate as $Y_{K, 1}(\n)$.

In the case $K = F$, we write $U_F^*(\m, \n)$, $U_F^*(\m(\aaa), \n)$, $U_{F, 1}^*(\n)$ for the intersections of the above groups with $G^*$.
\end{definition}

\begin{example}
The following three locally symmetric spaces are of particular importance in the sequel, so here we describe them explicitly (and record some of their other basic properties) for reference later in the paper.
\begin{enumerate}[(i)]
\item $Y_{\Q,1}(N)$ is the usual (open) modular curve of level $\Gamma_1(N)$. It has one connected component, isomorphic to $\Gamma_{1}(N)\backslash\uhp$.
\item The space $Y_{F,1}^*(\n)$ also has a single connected component, isomorphic to $\Gamma_{F,1}^*(\n)\backslash\uhs$, where
\begin{align*}
 \Gamma_{F,1}^*(\n) &\defeq G^*(F)_+ \cap U^*_{F,1}(\n)\\
 &= \left\{ \smallmatrd{a}{b}{c}{d} \in \SLt(\roi_F): c = 0, a = d = 1 \bmod \n\right\}.
\end{align*}
\item Since $\det(U_{F,1}(\n)) = \roihat^\times$, the space $Y_{F,1}(\n)$ has $h_F$ connected components, where $h_F$ is the class number of $F$. The identity component is isomorphic to $\Gamma_{F,1}(\n)\backslash\uhs$, where
\[\Gamma_{F,1}(\n) \defeq \GLt(F) \cap U_{F,1}(\n).\]
\end{enumerate}
\end{example}

If $N = \n\cap \Z$, then there are natural maps
\[Y_{\Q,1}(N) \labelrightarrow{\iota} Y_{F,1}^*(\n) \labelrightarrow{\jmath} Y_{F,1}(\n)\]
induced by the natural maps $\uhp \hookrightarrow \uhs$ and $\GLt(\A_\Q^f) \hookrightarrow G^*(\A_\Q^f) \hookrightarrow G(\A_\Q^f)$. The map $\iota$ is injective in most cases:

\begin{proposition}
 \label{prop:goodinclusion}
 If $\n$ is divisible by some integer $q \ge 4$, then $Y_{F, 1}^*(\n)$ is a smooth manifold, and
 \[ \iota: Y_{\Q, 1}(N) \hookrightarrow Y_{F, 1}^*(\n) \]
 is a closed immersion.
\end{proposition}

\begin{proof}
 First, the smoothness assertion. It suffices to prove that $\Gamma_{F, 1}^*(\n)$ has no non-trivial torsion elements. Since $\Gamma_1^*(\n)$ is a subgroup of $\SLt(\roi_F)$, any torsion element $\gamma$ must have eigenvalues $\zeta, \zeta^{-1}$ where $\zeta$ is a (non-trivial) root of unity, defined over an extension of $F$ of degree at most 2. Since $\zeta + \zeta^{-1} = a + d = 2 \bmod \n$, we conclude that $\n$ divides $\zeta + \zeta^{-1} - 2$. A case-by-case check shows that this implies $\zeta$ has order $2, 3, 4$ or $6$, and $\n$ must contain one of the integers $1$,$2$,$3$.

 Let us now prove the injectivity assertion. Let $z, z' \in \uhs$ be such that $\gamma z = z'$, for some $\gamma \in \Gamma_1^*(\n)$. Then $\gamma^{-1}\bar\gamma  z = z$, so either $\gamma^{-1}\bar\gamma = \mathrm{id}$, or $\gamma^{-1}\bar\gamma$ is a non-trivial torsion element in $\SLt(\roi_F)$. Since $\gamma$ is upper-triangular modulo some integer $q \ge 4$, the same is true of $\bar\gamma$ and thus also of $\gamma^{-1}\bar\gamma$; but we have just seen that $\Gamma_{F, 1}^*(q)$ has no torsion elements for $q \ge 4$.

 We can therefore conclude that $\gamma^{-1}\bar\gamma = \mathrm{id}$, in other words that $\gamma \in \Gamma_{F, 1}^*(\n) \cap \SLt(\Z) = \Gamma_{\Q, 1}(N)$. Hence $z = z'$ as elements of $Y_{\Q, 1}(N)$.
\end{proof}

\begin{remark-numbered}
 Henceforth, we will always assume that $\n$ is divisible by such a $q$, or, more generally, is small enough to avoid the possibility that these spaces are (non-smooth) orbifolds.
\end{remark-numbered}

In contrast, the composition $\jmath\circ\iota$ is \emph{never} injective, since $\smallmatrd{-1}{0}{0}{1} \in \Gamma_1(\n)$ preserves the image of $\uhs$ in $\uhp$, and acts on $\uhp$ by $x+iy \mapsto -x + iy$, so the points $-x +iy$ and $x + iy$ of $Y_{\Q, 1}(N)$ (which are distinct for generic $x, y$) are identified when mapped to $Y_{F,1}(\n)$. This failure of injectivity is a key reason for introducing the space $Y_{F,1}^*(\n)$. In fact, one can see directly that:

\begin{proposition}\label{prop:map-j}
The map $\jmath : Y_{F,1}^*(\n) \rightarrow Y_{F,1}(\n)$ has image equal to the identity component $\Gamma_{F,1}\backslash\uhs$. Its fibres are the orbits of the finite group $\left\{\smallmatrd{\epsilon}{0}{0}{1} \hspace{2pt} : \hspace{2pt} \epsilon \in \roi_F^\times\right\}$ acting on $\Gamma^*_{F, 1} \backslash \uhs$.
\end{proposition}

\subsection{Hecke correspondences}
\label{hecke operators}

We can define Hecke correspondences on the symmetric spaces $Y_{K, 1}(\n)$, for $\n$ an ideal of $\roi_K$, as follows. Firstly, we have diamond operators $\langle w \rangle$ for every $w \in (\roi_K / \n)^\times$, which define an action of $(\roi_K / \n)^\times$ on $Y_{K,1}(\n)$; this even extends to an action of the narrow ray class group modulo $\n$, although we shall not use this.

Secondly, let $\aaa$ be a square-free ideal of $\roi_K$. Consider the diagram
\[
\xymatrix{
& Y_K(1(\aaa),\n) \ar[ld]^{\pi_2} \ar[rd]^{\pi_1}&\\
Y_{K, 1}(\n)&& Y_{K, 1}(\n),
}
\]
where $\pi_1$ is the natural projection map, and $\pi_2$ is the `twisted' map given by the right-translation action of $\smallmatrd{\varpi}{}{}{1}$ on $\GLt(\A^f_K)$, where $\varpi \in \roikhat$ is any integral ad\`ele which generates the ideal $\aaa \roikhat$. (If $K = \Q$ and $\varpi = a$ is the positive integer generating $\aaa$, then this map $\pi_2$ corresponds to $z \mapsto z/a$ on $\uhp$.) We then define
\[(T_{\aaa})_* \defeq (\pi_2)_* \circ (\pi_1)^*\]
\[(T_{\aaa})^* \defeq (\pi_1)_* \circ (\pi_2)^*\]
as correspondences on $Y_{K, 1}(\n)$. When $\aaa$ divides the level $\n$, we denote these operators instead by $(U_\aaa)_*$ and $(U_\aaa)^*$. The definition may be extended to non-squarefree $\aaa$ in the usual way.

The same construction is valid for the more general symmetric spaces $Y_K(\m, \n)$, but it is no longer independent of the choice of generator $\varpi$ of $\aaa$ (it depends on the class of $\varpi$ modulo $1 + \m \roikhat$). We will only use this in the case where $\aaa$ is generated by a positive integer $a$, in which case we of course take $\varpi = a$. With this convention, the Hecke operators $(T_a)_*$ and $(T_a)^*$ for positive integers $a$ also make sense on the ``hybrid'' symmetric spaces $Y_{F, 1}^*(\m, \n)$.

\begin{remark-numbered}
 The maps $(T_\aaa)^*$ and $(U_\aaa)^*$ are perhaps more familiar, as their action on automorphic forms is given by simple formulae in terms of Fourier expansions, as we shall recall below. The lower-star versions $(T_\aaa)_*$ and $(U_\aaa)_*$ are the transpose of the upper-star versions with respect to Poincar\'e duality; this duality explains the key role played by $(U_p)_*$ in our norm relation computations.
\end{remark-numbered}

%
%
\subsection{Bianchi modular forms}
\label{bianchi modular forms}

We briefly recall the definition of Bianchi modular forms; for further details following our exact conventions, see \cite[\S 1]{Wil17}, or \cite{Gha99} for a more general treatment. As above, let $F$ be an imaginary quadratic field, and $U$ an open compact subgroup of $\GLt(\A^f_F)$. Then, for any $k \ge 0$, there is a finite-dimensional $\C$-vector space $S_{k,k}(U)$ of \emph{Bianchi cusp forms} of weight $(k, k)$ and level $U$, which are functions
\[ \Psi: \GLt(F) \backslash \GLt(\A_F) / U \longrightarrow V_{2k+2}^{(r)}(\C) \]
transforming appropriately under right-translation by the group $\C^\times \cdot \SUt(\C)$, and satisfying suitable harmonicity and growth conditions.

These forms can be described by an appropriate analogue of $q$-expansions (cf.~\cite[\S 1.2]{Wil17}). Let $e_F : \A_F/F \rightarrow \C^\times$ denote the unique continuous character whose restriction to $F \otimes \R \cong \C$ is
\[ x_\infty \longmapsto e^{2\pi i\mathrm{Tr}_{F/\Q}(x_\infty)}, \]
and let $W_\infty: \C^\times \to V_{2k+2}(\C)$ be the real-analytic function defined in 1.2.1(v) of \emph{op.cit.} (involving the Bessel functions $K_n$).

\begin{theorem}
 Let $\Psi$ be a Bianchi modular form of weight $(k,k)$ and level $U$. Then there is a \emph{Fourier--Whittaker expansion}
 \[\Psi\left(\matrd{\mathbf{y}}{\mathbf{x}}{0}{1}\right) = |\mathbf{y}|_{\A_F}\sum_{\zeta \in F^\times}W_f(\zeta \mathbf{y}_{f}, \Psi) W_{\infty}(\zeta \mathbf{y}_\infty) e_F(\zeta \mathbf{x}),\]
 where $W_{f}(-, \Psi)$, the ``Kirillov function'' of $\Psi$, is a locally constant function on $(\A_{F}^f)^\times$, with support contained in a compact subset of $\A_F^f$.
\end{theorem}

If $U = U_{F, 1}(\n)$ for some $\n$, then $W_f(-, \Psi)$ is supported in $\mathcal{D}^{-1}\roihat$. For $\m$ an ideal of $\roi_F$, we define a coefficient $c(\m, \Psi)$ as the value $W_{f}(\mathbf{y}_f, \Psi)$ for any $\mathbf{y}_f$ generating the fractional ideal $\mathcal{D}^{-1} \m \roihat$; this is independent of the choice of $\mathbf{y}_f$.

Exactly as for elliptic modular forms, the space $S_{k, k}(U_{F, 1}(\n))$ has an action of (commuting) Hecke operators $(T_\m)^*$ for all ideals $\m$; and if $\Psi$ is an eigenvector for all these operators, normalized such that $c(1, \Psi) = 1$, then the eigenvalue of the $\m$-th Hecke operator on $\Psi$ is $c(\m, \Psi)$. Moreover, the space $S_{k, k}(U_{F, 1}(\n))$ is a direct sum of ``new'' and ``old'' parts, and the Hecke operators $(T_\m)^*$ are simultaneously diagonalisable on the new part.

\subsection{The Asai $L$-function}
 \label{sect:asaiLfcn}
 We now define the principal object of study in this paper, the Asai $L$-function of a Bianchi eigenform.

 The space $S_{k, k}(U_{F, 1}(\n))$ has an action of diamond operators $\langle d \rangle$, for all $d \in (\roi_F / \n)^\times$; and on any Hecke eigenform $\Psi$ these act via a character $\varepsilon_{\Psi}: (\roi_F / \n)^\times \to \C^\times$. Let $\varepsilon_{\Psi, \Q}$ denote the restriction of this character to $(\Z / N\Z)^\times$.

 \begin{definition}
  \label{def:asaiLS}
  Let $\Psi$ be a normalized eigenform in $S_{k,k}(U_{F, 1}(\n))$, and $\chi$ a Dirichlet character of conductor $m$. Define the \emph{Asai $L$-function} of $\Psi$ by
  \[
   L^{\mathrm{As}}(\Psi, \chi, s) \defeq L^{(mN)}(\chi^2\varepsilon_{\Psi, \Q},2s-2k-2) \cdot \sum_{\substack{n \ge 1 \\ (m, n) = 1}}c(n\roi_F,\Psi)\chi(n)n^{-s},
  \]
  where $N = \n \cap \Z$ and $L^{(mN)}(-,s)$ is the Dirichlet $L$-function with its Euler factors at primes dividing $mN$ removed. If $\chi$ is trivial we write simply $L^{\mathrm{As}}(\Psi,s)$.
 \end{definition}

 This Dirichlet series is absolutely convergent for $\Re(s)$ sufficiently large\footnote{Since the eigenvalues $c(\mathfrak{l}, \Psi)$ for $\l$ prime satisfy $|c(\mathfrak{l}, \Psi)| \le 2N_{F/\Q}(\mathfrak{l})^{k/2 + 1}$ \cite[\S 11]{jacquetlanglands}, it suffices to take $\Re(s) > k+3$. This bound is not optimal, but is suffices for our purposes.}, and has meromorphic continuation to all $s \in \C$. (This is proved in \cite{Fli88} or \cite{Gha99} for $\chi$ trivial, and the result for general $\chi$ can be proved similarly.) For $s$ in the half-plane of convergence, it can be written as an Euler product
 \[ L^{\mathrm{As}}(\Psi, \chi, s) = \prod_{\text{$\ell$ prime}} L_\ell^{\mathrm{As}}(\Psi, \chi, s), \]
 where $L_\ell^{\mathrm{As}}(\Psi, \chi, s)$ depends only on $\chi(\ell)$ and the Hecke and diamond eigenvalues of $\Psi$ at the primes above $\ell$. If $\ell \mid m$ then $L_\ell^{\mathrm{As}}(\Psi, \chi, s) = 1$; if $\ell \nmid m$, then a case-by-case check shows that $L_\ell^{\mathrm{As}}(\Psi, \chi, s)$ has the form $P_\ell(\Psi, \ell^{-s} \chi(\ell))^{-1}$, where $P_\ell(\Psi, -)$ is a polynomial of degree $\le 4$.

 If $\Psi$ is a new eigenform, and $\chi$ is trivial, then  $L^{\mathrm{As}}(\Psi, \chi, s)$ coincides with the function denoted by $G(s, f)$ in \cite{Gha99}.

 \begin{lemma}
  Suppose $\Psi$ is a normalised eigenform of level $\n$ coprime to $p$, $\n' = \n \prod_{\mathfrak{p} \mid p} \mathfrak{p}$, and $\Psi_\alpha$ is a normalised eigenform of level $\n'$ such that $c(\m, \Psi_\alpha) = c(\m, \Psi)$ for all $\m$ coprime to $p$. Let $\alpha = c(p\roi_F, \Psi_\alpha)$. Then we have
  \[
   L^{\mathrm{As}}(\Psi_\alpha, s) = (1 - \alpha p^{-s})^{-1} L^{\mathrm{As}, (p)}(\Psi, s),
  \]
  and $L^{\mathrm{As}}(\Psi_\alpha, \chi, s) = L^{\mathrm{As}}(\Psi, \chi, s)$ for every non-trivial $\chi$ of $p$-power conductor.\qed
 \end{lemma}

\subsection{The primitive Asai $L$-function}
 \label{sect:primitiveL}

 The ``imprimitive'' Asai $L$-function of Definition \ref{def:asaiLS} is closely related to another $L$-function, the ``primitive'' Asai $L$-function. This is slightly more difficult to define, but in many ways more fundamental.

 \subsubsection*{Automorphic construction:} The Langlands $L$-group of $G = \operatorname{Res}_{F/\Q}\GLt$ is the semidirect product ${}^L G = (\GLt(\C) \times \GLt(\C)) \rtimes \operatorname{Gal}(\overline{\Q}/\Q)$, with the Galois group acting on $\GLt \times \GLt$ by permuting the factors via its quotient $\operatorname{Gal}(F/\Q)$. There is a 4-dimensional representation of ${}^L G$, the \emph{Asai representation} $r^{\mathrm{As}}$, whose restriction to $\GLt \times \GLt$ is the tensor product map $\GLt \times \GLt \to \operatorname{GL}_4$ (see e.g.~\cite[\S 0]{Fli88}).

 \begin{remark-numbered}
  The representation $r^{\mathrm{As}}$ of ${}^L G$ factors through the quotient ${}^L G^*$, which explains the prominent role the group $G^*$ plays in our constructions.
 \end{remark-numbered}

 If $\Pi$ is the automorphic representation of $\GLt(\A_F)$ generated by an eigenform $\Psi$ as above, then for each rational prime $\ell$, the local factors $\Pi_v$ for primes $v \mid \ell$ of $F$ determine (via the local Langlands correspondence for $\GLt$) a Weil--Deligne representation $w_{\Pi, \ell}$ of $\Ql$ with values in ${}^L G$. Moreover, a Dirichlet character $\chi$ modulo $m$ determines uniquely a character $\chi_\A = \prod_\ell \chi_\ell$ of $\A_\Q^\times / \Q^\times$ such that for $\ell \nmid m$, $\chi_\ell$ is unramified and maps a uniformiser to $\chi(\ell)$. We may interpret $\chi_\ell$ also as a character of the Weil group of $\Ql$, via the Artin reciprocity map (normalised to send uniformisers to geometric Frobenius elements).

 \begin{definition}
  \label{def:local-lfactor}
  We let $L_\ell^{\mathrm{As}}(\Pi, \chi, s)$ denote the local $L$-factor of the 4-dimensional Weil-Deligne representation $(r^{\mathrm{As}} \circ w_{\Pi, \ell}) \otimes \chi_\ell$, where $r^{\mathrm{As}}$ is the 4-dimensional tensor product representation of ${}^L G$. We define the \emph{primitive Asai $L$-function} of $\Pi$ by
  \[
   L^{\mathrm{As}}(\Pi, \chi, s) = \prod_{\text{$\ell$ prime}} L_\ell^{\mathrm{As}}(\Pi, \chi, s).
  \]
  If $\chi$ is trivial we write simply $L_\ell^{\mathrm{As}}(\Pi, s)$ and $L^{\mathrm{As}}(\Pi, s)$.
 \end{definition}

 \begin{remark-numbered} \
  \begin{enumerate}[(i)]
   \item Theorem 1.4 of \cite{Ram02} shows that $L^{\mathrm{As}}(\Pi, \chi, s)$ is an automorphic $L$-function: more precisely, there exists an automorphic representation $\mathrm{As}(\Pi)$ of $\operatorname{GL}_4(\A_{\Q})$ such that $L^{\mathrm{As}}(\Pi, \chi, s)$ coincides with the $L$-function of $\mathrm{As}(\Pi) \otimes \chi$, for every Dirichlet character $\chi$.
   \item The $L$-factor $L_\ell^{\mathrm{As}}(\Pi, \chi, s)$ can also be defined as the ``lowest common denominator'' of the values of a local zeta-integral, as in \cite{Fli88}. This alternative definition is known to be equivalent to Definition \ref{def:local-lfactor}: if $\ell$ is split in $F$, this equivalence is one of the defining properties of the local Langlands correspondence (see condition (2) in the introduction of \cite{harristaylor01}); if $\ell$ is inert or ramified in $F$ the equivalence is given by \cite[Theorem 4.2]{matringe10}.
  \end{enumerate}
 \end{remark-numbered}

 \subsubsection*{Galois representations:} This primitive $L$-function $L^{\mathrm{As}}(\Pi, \chi, s)$ can also be understood in terms of Galois representations, as follows. As we shall see in the next section, for an eigenform $\Psi$ as above, the coefficients $c(\n, \Psi)$ all lie in a number field $E \subset \C$. If $\mathfrak{p}$ is a prime of $E$, above some rational prime $p$, there is a unique semisimple 2-dimensional $E_\mathfrak{p}$-linear representation $V_{\mathfrak{p}}(\Pi)$  of $\operatorname{Gal}(\overline{\Q} / F)$, unramified outside $\n p$, on which geometric Frobenius at a prime $\mathfrak{l} \nmid \n p$ of $F$ has trace $c(\mathfrak{l}, \Psi)$. Then we may form the \emph{tensor induction} $\operatorname{As} V_{\mathfrak{p}}(\Pi)$, which is a 4-dimensional representation of $\operatorname{Gal}(\overline{\Q} / \Q)$, isomorphic as a representation of $\operatorname{Gal}(\overline{\Q} / F)$ to the tensor product of $V_{\mathfrak{p}}(\Pi)$ and its Galois conjugate.

 The local Euler factor at $\ell$ of this Galois representation is given by
 \begin{equation}
  \label{eq:localEF}
  L_\ell(\operatorname{As} V_{\mathfrak{p}}(\Pi) \otimes \chi, s) \coloneqq \det\left( 1 - \ell^{-s} \operatorname{Frob}_{\ell}^{-1}:  \left(\operatorname{As} V_{\mathfrak{p}}(\Pi) \otimes \chi\right)^{I_\ell} \right)^{-1},
 \end{equation}
 where $I_\ell$ is the inertia group at $\ell$, $\operatorname{Frob}_{\ell}$ is an arithmetic Frobenius element, and $\mathfrak{p}$ is any prime not dividing $\ell$. If $V_{\mathfrak{p}}(\Pi)$ satisfies local-global compatibility at the primes above $\ell$ (which is expected, and known in many cases by the results of \cite{mok14} and \cite{varma}), then this local $L$-factor $L_\ell(\operatorname{As} V_{\mathfrak{p}}(\Pi) \otimes \chi, s)$ coincides with the automorphic local $L$-factor $L_\ell^{\mathrm{As}}(\Pi, \chi, s)$.

 \subsubsection*{Relation to the imprimitive $L$-function:} The imprimitive $L$-function $L^{\mathrm{As}}(\Psi, \chi, s)$ can be seen as an ``approximation'' to the primitive $L$-function, in the following sense:

 \begin{proposition}
  We have
  \[
    L^{\mathrm{As}}(\Psi, \chi, s) = L^{\mathrm{As}}(\Pi, \chi, s) \prod_{\ell \mid mN} C_\ell(\Psi, \chi, s),
  \]
  where the local error terms $C_\ell(\Psi, \chi, s)$ are polynomials in $\ell^{-s}$, of degree $\le 4$. In particular, the primitive and imprimitive $L$-functions have the same local factors at all but finitely many primes.
 \end{proposition}

 \begin{proof}
  The equality of the local factors for $\ell \nmid mN$ is Proposition 1 of \cite{Gha99}; so it remains to show that if $\ell \mid mN$, the ratio $\frac{L_\ell^{\operatorname{As}}(\Psi, \chi, s)}{L_\ell^{\operatorname{As}}(\Pi, \chi, s)}$ is a polynomial in $\ell^{-s}$ of degree $\le 4$.

  If $\ell \mid m$, then $L_\ell^{\operatorname{As}}(\Psi, \chi, s)$ is identically 1. The same holds if $\Pi_\ell$ is supercuspidal at some prime above $\ell$, since in this case $c(n \roi_F, \Psi)$ is zero for every $n$ divisible by $\ell$. Since $L_\ell^{\operatorname{As}}(\Pi, \chi, s)$ is (by definition) the reciprocal of a polynomial of degree $\le 4$, the result is automatic in these cases. This leaves the primes $\ell$ such that $\ell \mid N$, $\ell \nmid m$, and the local factors of $\Pi$ at primes above $\ell$ are principal series or special, which can be verified in a case-by-case check.
 \end{proof}

 Perhaps surprisingly, the local error terms $C_\ell(\Psi, \chi, s)$ at the primes $\ell \mid mN$ may be non-trivial, even if $\chi = 1$ and $\Psi$ is the unique \emph{new} eigenform generating $\Pi$. In Galois-theoretic terms, this corresponds to the fact that the inertia invariants of $\operatorname{As} V_{\mathfrak{p}}(\Pi)$ at a prime $\ell$ may be strictly larger than the tensor product of the inertia invariants of $V_{\mathfrak{p}}(\Pi)$ at the primes above $\ell$. However, there are some simple criteria that are sufficient to rule this out:
 \begin{itemize}
  \item If $\ell \mid m$ and $\ell \nmid ND$, then both $L_\ell^{\mathrm{As}}(\Psi, \chi, s)$ and $L_\ell^{\mathrm{As}}(\Pi, \chi, s)$ are identically 1, so in particular $C_\ell(\Psi, \chi, s) = 1$.
  \item If $\ell \nmid m$, $\ell$ is split in $F$, and only one of the two primes above $\ell$ divides $\n$, then a direct computation shows that $C_\ell(\Psi, \chi, s) = 1$ (compare Remark 2.7.1 of \cite{KLZ17} in the Rankin--Selberg case).
 \end{itemize}


\subsection{The Coates--Perrin-Riou conjecture}
\label{sec:coates-perrin-riou}

In \cite{coatesperrinriou89} and \cite{coates89}, a general conjecture is formulated predicting the existence of $p$-adic $L$-functions associated to any motive over $\Q$ whose $L$-function has at least one critical value. In this section, we shall explain what this conjecture predicts for the (conjectural) Asai motive associated to a Bianchi eigenform.

Let us fix an automorphic representation $\Pi$ of $\GLt(\A_F)$, generated by some Bianchi eigenform $\Psi$ of weight $(k, k)$ with coefficients in a number field $E$. We shall assume (in this section only) that there exists a Chow motive $M^{\mathrm{As}}(\Pi)$ over $\Q$, with coefficients in $E$, whose $L$-function is $L^{\mathrm{As}}(\Pi, s)$, and whose $\mathfrak{p}$-adic realisation is $\operatorname{As} V_{\mathfrak{p}}(\Pi)$, for every prime $\mathfrak{p} $ of $E$. (Strictly speaking the conjectures are formulated only for motives with coefficients in $\Q$, but the extension to general $E$ is immediate.)

We shall apply the conjectures of \cite{coates89} to the motive $M = M^{\mathrm{As}}(\Pi)(1)$. Note that $M$ has weight $2k$, and its Hodge decomposition at $\infty$ is given by $\operatorname{dim} M^{(-1, 2k+1)} = \operatorname{dim} M^{(2k+1, -1)} = 1$, $\operatorname{dim} M^{(k,k)} = 2$. Moreover, complex conjugation acts as $-1$ on $M^{(k,k)}$, so $d^+(M) = 1$ and $d^-(M) = 3$. Thus $s = 0$ is a critical value of $L(M(j)(\chi), s) = L^{\mathrm{As}}(\Pi, \chi^{-1}, s + j)$, for integers $0 \le j \le k$ and Dirichlet characters $\chi$ such that $(-1)^j\chi(-1) = 1$.\footnote{These are all the critical values left of the centre of symmetry of the functional equation. Right of this line, we also have critical values for $k+1 \le j \le 2k+2$ with $(-1)^j\chi(-1) = -1$.} We choose our square root of $-1$ (the $\rho$ of \emph{op.cit.}) to be $+i$.

We shall also assume that $p$ is unramified in $F$ and $\Pi$ has conductor coprime to $p$, and is ordinary at the primes above $p$, so there is a unique eigenvalue $\alpha$ of $\operatorname{Frob}_p^{-1}$ on $\operatorname{As} V_{\mathfrak{q}}(\Pi)$, for $\mathfrak{q}$ not dividing $p$, such that $\alpha$ is a unit at $\mathfrak{p}$. We shall take $\chi$ to be a Dirichlet character of $p$-power conductor.

\begin{proposition}
 In the notation of \emph{op.cit.}, for $(j, \chi)$ as above we have
 \begin{align*}
  \mathcal{L}_\infty^{(\rho)}(M(j)(\chi)) &= 2(2 \pi i)^{-1-j} j!, \\
  \mathcal{L}_p^{(\rho)}(M(j)(\chi)) &=
  \begin{cases}
  (1 - \tfrac{p^j}{\alpha}) \cdot (1 - \tfrac{\alpha}{p^{1+j}})^{-1}
  & \text{if $\chi$ is trivial}, \\[2mm]
  G(\chi) \left( \tfrac{p^{j}}{\alpha}\right)^{r}
  & \text{if $\chi$ has conductor $p^r > 1$},
  \end{cases}
 \end{align*}
 where $G(\chi) \coloneqq \sum_{a \in (\Z/p^r)^\times} \chi(a) e^{2\pi i a / p^r}$.\qed
\end{proposition}

The conjecture of Coates and Perrin-Riou thus predicts that there should exist a non-zero constant $\Omega_M$, and a $p$-adic measure $\mu$ on $\Zp^\times$, such that
\[ \int_{\Zp^\times} x^j \chi(x)\, \mathrm{d}\mu(x) =  \mathcal{L}_p^{(\rho)}(M(j)(\chi)) \cdot \frac{j! L^{\mathrm{As}, (p)}(\Pi, \chi^{-1}, j+1)}{(2 \pi i)^{(1 + j)} \Omega_M}, \]
for $(j, \chi)$ as above. Stated in this form, the conjecture is independent of the existence of the motive $M^{\mathrm{As}}(\Pi)$.


\subsection{The modular symbol attached to $\Psi$}
\label{sect:modsymb}

The Bianchi modular forms we consider in this paper are \emph{cohomological}: they contribute to the cohomology of local systems on $Y_{F, 1}(\n)$, as we shall now explain.

\begin{definition}
 Let $E$ be any field extension of $F$, and let $k \ge 0$. We define the left $E[\GLt(F)]$-module $V_{kk}(E) \defeq V_k^{(\ell)}(E)\otimes_{E} V_k^{(\ell)}(E)^\sigma$, where $\gamma \in \GLt(F)$ acts in the usual way on the first component and via its complex conjugate $\gamma^{\sigma}$ on the second component. Via this action, the space $V_{kk}(E)$ gives rise to a local system of $E$-vector spaces on $Y_{F,1}(\n)$, which we also denote by $V_{kk}(E)$.
\end{definition}

\begin{theorem}[Eichler--Shimura--Harder] There is a canonical, Hecke-equivariant injection
\[S_{k,k}(U_{F,1}(\n)) \hookrightarrow \h^1_{\mathrm{c}}(Y_{F,1}(\n),V_{kk}(\C)). \]
 Moreover, if $\Psi \in S_{k,k}(U_{F,1}(\n))$ is a normalised eigenform, this map induces isomorphisms of 1-dimensional $\C$-vector spaces
\[ S_{k,k}(U_{F,1}(\n))[\Psi] \xrightarrow{\cong} \h^1_{\mathrm{c}}(Y_{F,1}(\n),V_{kk}(\C))[\Psi] \xrightarrow{\cong} \h^1(Y_{F,1}(\n),V_{kk}(\C))[\Psi].
\]
\end{theorem}

\begin{proof}
This was initially proved in \cite{Har87}. A treatment closer to our conventions is \cite{Hid94}, where Proposition 3.1 gives an isomorphism between the cusp forms and cuspidal cohomology, the injection of cuspidal cohomology into $\h_{\mathrm{c}}^1$ is explained at the start of \S5, and the 1-dimensionality of each of these spaces (hence the isomorphism) is explained in \S8.
\end{proof}

If $\Psi \in S_{k,k}(U_{F,1}(\n))$, write $\omega_{\Psi}$ for its image in $\h^1_{\mathrm{c}}(Y_{F,1}(\n),V_{kk}(\C))$. This can be described concretely as the class of a harmonic $V_{kk}$-valued differential form constructed from $\Psi$; for a summary of the construction using our conventions, see \cite[\S2.4]{Wil17}.

We now consider integral structures at $p$. Let $\Psi \in S_{kk}(U_{F, 1}(\n))$ be a normalised eigenform, $E$ a finite extension of $F$ containing the Hecke eigenvalues of $\Psi$, and $\mathscr{P}$ a prime of $E$ above $p$. We write $R' = \roi_{F, (\mathscr{P})}$ for the valuation ring of $E$ at $\mathscr{P}$, and $R$ for its completion (so $R$ is a finite-rank free $\Zp$-algebra). The same construction as before gives a finite-rank free $R'$-module $V_{kk}(R')$, with a left action of $\GLt(\roi_{F, (p)})$, whose base-extension to $E$ is $V_{kk}(E)$.

If $\widehat{\roi}_{F, (p)} = \prod_{v \mid p} \roi_{F, v} \times \prod_{v \nmid p}' F_v$ denotes the ring of finite ad\`eles of $F$ integral above $p$, then the natural map
\[ \GLt(\roi_{F, (p)}) \backslash \left[ \GLt(\widehat{\roi}_{F, (p)}) \times \uhs \right] / U \longrightarrow Y_F(U), \]
is a bijection for any level $U$ (by the weak approximation theorem). Thus $V_{kk}(R')$ gives a local system of $R'$-modules on $Y_F(U)$ for any sufficiently small $U$, whose base-extension to $E$ is $V_{kk}(E)$ as previously defined.

Since every connected component of $Y_F(U)$ is non-compact, its compactly-supported cohomology (with coefficients in any locally constant sheaf) vanishes in degree 0. From the cohomology long exact sequence associated to multiplication by a non-zero integer $m$, we conclude that its degree 1 compactly-supported cohomology with coefficients in a torsion-free sheaf is torsion-free. Hence we can regard $\h^1_{\mathrm{c}}(Y_{F,1}(\n),V_{kk}(R'))$ as an $R'$-lattice in $\h^1_{\mathrm{c}}(Y_{F,1}(\n),V_{kk}(E))$, preserved by the action of the Hecke operators $(T_{\aaa})^*, (U_{\aaa})^*$.

\begin{definition}
 We let $\h^1_{\mathrm{c}}(Y_{F,1}(\n),V_{kk}(R'))[\Psi]$ be the intersection of the $\Psi$-eigenspace $\h^1_{\mathrm{c}}(Y_{F,1}(\n),V_{kk}(E))[\Psi]$ with the lattice $\h^1_{\mathrm{c}}(Y_{F,1}(\n),V_{kk}(R'))$.
\end{definition}

\begin{proposition}
 There exists a complex period $\Omega_{\Psi} \in \C^\times$, uniquely determined up to multiplication by $(R')^\times$, such that the quotient
 \[
 	\phi_{\Psi} \defeq \omega_\Psi/\Omega_\Psi
 \]
 is an $R'$-basis of $\h^1_{\mathrm{c}}(Y_{F,1}(\n),V_{kk}(R'))[\Psi]$.
\end{proposition}

\begin{proof}
 Since compactly-supported cohomology commutes with flat base extension, it follows from the Eichler--Shimura--Harder theorem that $\h^1_{\mathrm{c}}(Y_{F,1}(\n),V_{kk}(E))[\Psi]$ is one-dimensional over $E$. If $\phi_\Psi$ is any basis of this space, then it is also a $\C$-basis of $\h^1_{\mathrm{c}}(Y_{F,1}(\n),V_{kk}(\C))[\Psi]$, so $\omega_{\Psi}$ must be a $\C^\times$-multiple of $\phi_\Psi$.

 As $R'$ is a discrete valuation ring with field of fractions $E$, the lattice $\h^1_{\mathrm{c}}(Y_{F,1}(\n),V_{kk}(R'))[\Psi]$ must be free of rank 1 over $R'$. We can therefore choose $\phi_\Psi$ to be a generator of this module, and this determines $\phi_\Psi$ (and hence $\Omega_{\Psi}$) uniquely up to units in $R'$.
\end{proof}

\begin{remarks}\
\begin{enumerate}[(i)]
\item Our normalisation of the period $\Omega_{\Psi}$ is a little different from that used by Hida in \cite[\S 8]{Hid94}: we have used the lattice given by $\h^1_{\mathrm{c}}$ with integral coefficients, while Hida uses the ordinary $\h^1$ with integral coefficients. Hence our $\phi_{\Psi}$ may not map to an $R'$-basis of $\h^1(Y_{F,1}(\n), V_{kk}(R'))[\Psi]$. We expect that the difference between these two $R'$-lattices ``detects'' congruences modulo $\mathscr{P}$ between $\Psi$ and Eisenstein series, but we have not checked this.

\item There is a convenient explicit presentation for $\h^1_{\mathrm{c}}(Y_{F, 1}(\n), M)$, for very general local systems $M$, using ``modular symbols'' (see \cite[Lemma 8.4]{BW17}, generalising \cite[Proposition 4.2]{AS86}). From this description the torsion-freeness of the cohomology, and its compatibility with base-extension, is immediate.
\end{enumerate}
\end{remarks}

The pullback $\phi_{\Psi}^* \defeq \jmath^*(\phi_{\Psi})$ lies in $\h^1_{\mathrm{c}}(Y^*_{F,1}(\n), V_{kk}(R'))$. It is this class which we shall use in the definition of our $p$-adic Asai $L$-function.

%
%
\section{Siegel units and weight 2 Asai--Eisenstein elements}
\label{siegel units}

\subsection{Modular units}

Let $U \subset \GLt(\widehat{\Z})$ be an open compact subgroup, with associated symmetric space $Y_\Q(U)$. In this section, we work exclusively over $\Q$, so we shall drop the subscript $\Q$ from the notation. As is well known, the manifolds $Y(U)$ are naturally the complex points of algebraic varieties defined over $\Q$.

\begin{definition}
A \emph{modular unit} on $Y(U)$ is an element of $\roi(Y(U))^\times$, that is, a regular function on $Y(U)$ with no zeros or poles. (This corresponds to a rational function on the compactification $X(U)$ whose divisor is supported on the cusps).
\end{definition}
Modular units are \emph{motivic} in the sense that there are \emph{realisations} of modular units in various cohomology theories. In particular, to a modular unit $\phi\in \roi(Y_1(N))^\times$ one can attach:
\begin{itemize}
\item its \emph{de Rham} realisation $C_{\mathrm{dR}}(\phi) \in \h^1_{\mathrm{dR}}(Y_1(N),\Q)$, which is the class of the differential form $d\log \phi = \frac{\mathrm{d}\phi}{\phi}$;
\item its \emph{Betti} realisation $C(\phi) \in \h^1(Y_1(N),\Z)$, which is the pullback along $\phi: Y_1(N)(\C) \to \C^\times$ of the generator $C$ of $\h^1(\C^\times, \Z) \cong \Z$ that pairs to $1$ with the homology class of a positively-oriented loop around 0.
\end{itemize}

These are closely related:

\begin{proposition}
The canonical comparison isomorphism
\[
	  \h^1\left(Y_1(N),\Z\right)\otimes_{\Z} \C \xrightarrow{\cong} \h^1_{\mathrm{dR}}(Y_1(N),\Q)\otimes_{\Q} \C
\]
maps $C(\phi)$ to $\tfrac{1}{2\pi i} C_{\mathrm{dR}}(\phi)$.
\end{proposition}

\begin{proof}
 The comparison isomorphism identifies the de Rham cohomology class of an $n$-form $\omega$ with the Betti cohomology class mapping an $n$-simplex $\delta$ to $\int_{\delta} \omega$. By construction, this isomorphism is compatible with the pullback maps in the two cohomology theories attached to a smooth map of manifolds.

 In our case, the modular unit $\phi$ gives a map $Y_1(N)(\C) \to \C^\times$. By definition, $C(\phi)$ is the pullback by $\phi$ of the generator $C$ of $\h^1(\C^\times, \Z)$; and $C_{\mathrm{dR}}(\phi)$ is the pullback of the class of the differential $\tfrac{dT}{T}$, where $T$ is the coordinate on $\C$. If $\delta$ denotes a positively-oriented loop around 0, then $\int_{\delta} \tfrac{\mathrm{d}T}{T} = 2\pi i$. Since $C$ pairs to $1$ with $\delta$ (by definition), the result follows.
\end{proof}

\begin{remark}
 It is conventional to ``normalize away'' the factor of $2\pi i$ by defining $C(\phi)$ as an element of $\h^1(Y_1(N), \Z(1))$, where $\Z(1) = 2\pi i\cdot \Z \subset \C$. However, this formalism does not work well when considering the cohomology of non-algebraic manifolds, so we shall not use it here.
\end{remark}

\subsection{Eisenstein series}
\label{eisenstein series}
The de Rham realisations of modular units give rise to weight 2 Eisenstein series in the de Rham cohomology. In the next section, we'll exhibit a canonical system of modular units -- the \emph{Siegel units} -- whose de Rham realisations can be written down very explicitly in terms of the following Eisenstein series.

\begin{definition}[{cf.~\cite[\S3]{Kat04}}]
 Let $\tau \in \uhp$, $k$ an integer $\ge 2$, and $\beta \in \Q/\Z$, with $\beta \ne 0$ if $k = 2$. Define
 \[F^{(k)}_{\beta}(\tau) \defeq \frac{(k-1)!}{(-2\pi i)^k}\sideset{}{'}\sum_{(m,n)\in\Z^2}\frac{e^{2\pi i \beta m}}{(m\tau + n)^k},\]
 where the prime denotes that the term $(m,n) = (0,0)$ is omitted. This is a modular form of weight $k$ and level $\Gamma_1(N)$, for any $N$ such that $N\beta = 0$.
\end{definition}

(Kato defines a slightly more general class of Eisenstein series $F^{(k)}_{\alpha,\beta}$; we shall only use the case $\alpha = 0$, so we drop it from the notation.)

For the $L$-value calculations in the appendix, we shall need to use the fact that this series is a specialisation of a real-analytic family. For $\Re(s) \ge 1 - \tfrac{k}{2}$, we define
		\begin{equation}\label{eqn:Eis}
			E^{(k)}_{\beta}(\tau, s) \defeq \frac{\Gamma(s+k)}{(-2\pi i)^k\pi^{s}}\sideset{}{'}\sum_{(m,n)\in\Z^2}\frac{\mathrm{Im}(\tau)^s}{(m\tau + n + \beta)^k|m\tau + n + \beta|^{2s}},
		\end{equation}
where the prime denotes that the term $(m,n) = (0,0)$ is omitted if $\beta = 0$ (but included otherwise). This has analytic continuation to all $s\in \C$, and we have
\[ F^{(k)}_{\beta}(\tau) = E^{(k)}_{\beta}(\tau, 1-k).\]
(See e.g.\ \cite[4.2.2(iv)]{LLZ14}.)

\subsection{Siegel units}

\begin{definition}\mbox{~}
 For $N \ge 1$, and $c > 1$ an integer coprime to $6N$, let
 \[ \subc g_N \in \roi(Y_1(N))^\times \]
 be Kato's Siegel unit (the unit denoted by $\subc g_{0,1/N}$ in the notation of \cite[\S 1]{Kat04}).
\end{definition}

Slightly abusively, we shall use the same symbol $\subc g_N$ for the pullback of this unit to $Y(M, N)$, for any $M \ge 1$ (we shall only need this when $M \mid N$).

As in \emph{op.cit.}, we note that if $c, d$ are two integers that are both $>1$ and coprime to $6N$, then we have the identity
\begin{equation}
 \label{eq:cdsymmetry}
 (d^2 - \langle d \rangle)\subc g_N = (c^2 - \langle c \rangle){}_d g_N.
\end{equation}
It follows that the dependence on $c$ may be removed after extending scalars to $\Q$: there is an element $g_N \in \roi(Y_1(N))^\times \otimes \Q$ such that $\subc g_N = (c^2 - \langle c \rangle) \cdot g_N$ for any choice of $c$.

\begin{proposition} \
 \label{siegel unit properties}
 \begin{itemize}
  \item[(i)] The Siegel units are \emph{norm-compatible}, in the sense that if $N'|N$ and $N$ and $N'$ have the same prime divisors, then under the natural map
  \[\mathrm{pr}: Y(M,N) \longrightarrow Y(M,N')\]
  we have
  \[(\mathrm{pr})_*(\subc g_N) = \subc g_{N'}.\]
  \item[(ii)] The de Rham realisation of $g_N$ is the Eisenstein series
  \[
   d\log( g_N)(\tau) = -2\pi i \, F^{(2)}_{1/N}(\tau)\, \mathrm{d}\tau.
  \]
 \end{itemize}
\end{proposition}

\begin{proof}
 The first part is proved in \cite{Kat04}, Section 2.11. The second part is Proposition 3.11(2) \emph{op.cit.}.
\end{proof}

One important use of Siegel units comes in the construction of \emph{Euler systems}; for example, see \cite{Kat04},\cite{LLZ14}, and \cite{KLZ17}. The basic method in each of these cases is similar; one takes cohomology classes attached to Siegel units under the realisation maps and pushes them forward to a different symmetric space, then exploits the norm compatibility to prove norm relations for these cohomology classes. We will do something similar in the Betti cohomology. In particular, we make the following definition:

\begin{definition}
 Let $\subc C_{N} \defeq C(\subc g_{N}) \in \h^1(Y_1(N),\Z)$ be the Betti realisation of $\subc g_{N}$.
\end{definition}

From Proposition \ref{siegel unit properties}(i), we see that if $p \mid N$, the classes ${}_c C_{Np^r}$ for $r \ge 0$ are compatible under push-forward, and define a class
\[ \subc C_{Np^\infty} \in \varprojlim_r \h^1(Y_1(Np^r),\Z).\]

 \begin{lemma}
  \label{lemma:siegel-unit-conjugation}
  If $\rho$ denotes the involution of $Y_1(N)$ corresponding to $x + iy \mapsto -x+iy$ on $\uhp$, then $\rho^*\left({}_c C_N\right) = -{}_c C_N$.
 \end{lemma}

 \begin{proof}
  The variety $Y_1(N)$ has a canonical model over $\Q$ for which $\rho$ corresponds to complex conjugation on $Y_1(N)(\C)$. Since the Siegel units are defined over $\Q$ in this model, they intertwine the action of $\rho$ on $Y_1(N)(\C)$ with complex conjugation on $\C^\times$. So it suffices to check that complex conjugation acts as $-1$ on $\h^1(\C^\times, \Z)$, which is clear.
 \end{proof}

\subsection{Asai--Eisenstein elements in weight 2}

Now let $F$ be an imaginary quadratic field, and $\n$ an ideal of $\roi_F$ divisible by some integer $\ge 4$. Recall that we have
\[ Y_{F, 1}^*(\n) = \Gamma_{F, 1}^*(\n) \backslash \uhs, \]
and that we showed in Proposition \ref{prop:goodinclusion} that the natural map
\[\iota : Y_{\Q,1}(N) \hookrightarrow Y_{F,1}^*(\n)\]
is a closed immersion. We also need the following map:

\begin{definition}
 \label{def:kappaam}
Let $m\ge 1$ and $a \in \roi_F$. Consider the map
\[ \kappa_{a/m}: Y_{F, 1}^*(m^2 \n) \to Y_{F, 1}^*(\n) \]
given by the left action of $\smallmatrd{1}{a/m}{0}{1} \in \SLt(F)$ on $\uhs$. This is well-defined\footnote{
Note that $\kappa_{a/m}$ is \emph{not} in general well-defined on $Y_{F,1}(m^2\n)$, since \[\smallmatrd{1}{a/m}{0}{1}\smallmatrd{-1}{0}{0}{1}\smallmatrd{1}{-a/m}{0}{1} = \smallmatrd{-1}{2a/m}{0}{1},\] which is not in $\Gamma_{F,1}(\n)$ if $2a \notin m\roi_F$.}, since it is easy to see that
\[\matrd{1}{a/m}{0}{1} \Gamma_{F, 1}^*(m^2\n) \matrd{1}{-a/m}{0}{1} \subset \Gamma_{F, 1}^*(\n), \]
and it depends only on the class of $a$ modulo $m\roi_F$.
\end{definition}

The elements we care about are the following. Assume that $\n$ is divisible by some integer $q \ge 4$, as above. Then we have maps
\[ Y_{\Q, 1}(m^2 N) \labelrightarrow{\iota} Y_{F, 1}^*(m^2 \n)  \labelrightarrow{\kappa_{a/m}} Y_{F, 1}^*(\n). \]
Moreover, writing $\h_*^{\mathrm{BM}}$ for Borel--Moore homology (homology with non-compact supports), there are isomorphisms
\begin{equation}
 \label{eq:BMhomology}
 \h^1(Y_{\Q, 1}(m^2 N), \Z) \cong \h_1^{\mathrm{BM}}(Y_{\Q, 1}(m^2 N), \Z), \quad \h^2(Y_{F, 1}(m^2 \n), \Z) \cong \h_1^{\mathrm{BM}}(Y_{F, 1}(m^2 \n), \Z),
\end{equation}
given by cap-product with the fundamental classes in $\h_2^{\mathrm{BM}}(Y_{\Q, 1}(m^2 N), \Z)$ and $\h_3^{\mathrm{BM}}(Y^*_{F, 1}(m^2 \n), \Z)$ respectively. Since Borel--Moore homology is covariantly functorial for proper maps, we can therefore define a pushforward map $\iota_*: \h^1(Y_{\Q, 1}(m^2 N), \Z) \to \h^2(Y^*_{F, 1}(m^2 \n), \Z)$.

\begin{definition}
For $a \in \roi_F/m\roi_F$, $m \ge 1$, and $c > 1$ coprime to $6 m N$, define
\[ \subc\Xi_{m,\n,a} \in \h^2\left(Y_{F,1}^*(\n),\Z\right) \]
to be the image of $\subc C_{m^2N}$ under the map $(\kappa_{a/m})_* \circ \iota_*$.
\end{definition}

\begin{proposition}
 \label{prop:xi-sign}
 We have $\smallmatrd{-1}{}{}{1}^* \cdot \subc\Xi_{m,\n,a} = \subc\Xi_{m,\n,-a}$.
\end{proposition}

\begin{proof}
 It is clear that $\smallmatrd{-1}{}{}{1} \circ \kappa_{a/m} = \kappa_{-a/m} \circ \smallmatrd{-1}{}{}{1}$ as maps $Y_{F, 1}^*(m^2\n) \to Y_{F, 1}^*(\n)$. Moreover, the action of $\smallmatrd{-1}{}{}{1}$ on $Y_{F, 1}^*(m^2 \n)$ preserves the image of $Y_{\Q,1}(m^2 N)$, and the involution of $Y_{\Q,1}(m^2 N)$ it induces is the involution $\rho$ of Lemma \ref{lemma:siegel-unit-conjugation}, which acts as $-1$ on ${}_c C_{m^2 N}$. Finally, we have $\smallmatrd{-1}{}{}{1}^* \circ \iota_* = -\iota_* \circ \rho^*$, because $\smallmatrd{-1}{}{}{1}$ preserves the orientation of $Y^*_{F, 1}(\n)$ and thus acts as $+1$ on the fundamental class in $\h_3^{\mathrm{BM}}$, but $\rho$ reverses the orientation of $Y_{\Q, 1}(m^2N)$ and thus acts as $-1$ on the fundamental class.
\end{proof}

\begin{definition}
Define
\[\subc \Phi_{\n,a}^r \in \h^2\Big(Y_{F, 1}^*(\n),\Z\Big) \otimes_{\Z}\Zp[(\Z/p^r)^\times]\]
by
\[\subc\Phi_{\n,a}^r \defeq \sum_{t \in (\Z/p^r)^\times} \subc\Xi_{p^r,\n,at}\otimes[t].\]
\end{definition}

\begin{lemma}\label{lemma:level-compat}
 If $\n \mid \n'$ are two ideals of $\roi_F$ with the same prime factors, then pushforward along the map $Y_{F, 1}(\n') \to Y_{F, 1}(\n)$ sends $\subc \Phi_{\n',a}^r$ to $\subc \Phi_{\n,a}^r$ (for any valid choices of $c$, $a$, $r$).
\end{lemma}

\begin{proof}
 This is immediate from the norm-compatibility of the Siegel units; compare \cite[Theorem 3.1.2]{LLZ14}.
\end{proof}

We now come to one of the key theorems of this paper, which shows that if $m = p^r$ with $r$ varying, then the above elements fit together $p$-adically into a compatible family. We now impose the assumption that $\n$ is divisible by all primes $v \mid p$ of $\roi_F$.

\begin{theorem}\label{norm relation}
Let $r\ge 1$, let $a$ be a generator of $\roi_F/(p\roi_F+\Z)$, and let
\[\pi_{r+1}: \h^2(Y_{F, 1}^*(\n),\Z) \otimes_{\Z}\Zp[(\Z/p^{r+1})^\times] \longrightarrow \h^2(Y_{F, 1}^*(\n),\Z) \otimes_{\Z}\Zp[(\Z/p^r)^\times]\]
denote the map that is the identity on the first component and the natural quotient map on the second component. Then we have
\[\pi_{r+1}( \subc\Phi_{\n,a}^{r+1}) = (U_p)_* \cdot \subc\Phi_{\n,a}^r,\]
where the Hecke operator $(U_p)_*$ acts via its action on $\h^2(Y_{F, 1}^*(\n),\Z)$. Similarly, when $r = 0$ we have
\[ \pi_{1}( \subc\Phi_{\n,a}^{1}) = \left( (U_p)_* - 1 \right) \cdot \subc\Phi_{\n,a}^0.\]
\end{theorem}

\begin{remark-numbered}
 Before embarking on the proof, which will occupy the next section of the paper, we pause to give a brief description of how this is important for the construction of $p$-adic Asai $L$-functions, in the simplest case of a Bianchi eigenform $\Psi$ of weight $(0,0)$. Define $e_{\mathrm{ord},*} \defeq \lim_{n\rightarrow\infty}(U_p^{n!})_*$ to be the ordinary projector on cohomology with $\Zp$ coefficients associated to $(U_p)_*$, so that $(U_p)_*$ is invertible on the space $e_{\mathrm{ord},*}\h^2(Y_1(\n),\Zp)$. Given the theorem, we see that the collection
 \[
  [(U_p)_*^{-r} e_{\mathrm{ord},*} \cdot \subc\Phi_{\n,a}^r]_{r\ge 1}
 \]
 forms an element $\subc \Phi_{\n,a}^\infty$ in the inverse limit
 \[
 e_{\mathrm{ord},*}\h^2(Y_{F,1}^*(\n),\Zp) \otimes \Zp[[\Zp^\times]].
 \]
 As in \S \ref{sect:modsymb}, the eigenform $\Psi$ gives rise to a cohomology class $\phi_\Psi^* \in \h^1_{\mathrm{c}}(Y_{F, 1}^*(\n), R)$, where $R$ is the ring of integers of some finite extension of $\Qp$. If the $(U_p)^*$-eigenvalue of $\Psi$ is a unit in $R$, then the linear functional on $\h^2(Y_{F,1}^*(\n),\Zp) \otimes R$ given by pairing with $\phi_\Psi^*$ factors through the projector $e_{\mathrm{ord},*}$. Hence pairing $\subc \Phi_{\n,a}^\infty$ with $\phi_\Psi^*$ gives a measure on $\Zp^\times$ with values in $R$. This will be our $p$-adic $L$-function. By construction, its values at finite-order characters are given by integrating $\Psi$ against linear combinations of Eisenstein series on $Y_{\Q,1}(m^2 N)$; and these will turn out to compute the special values of the Asai $L$-function.
\end{remark-numbered}

%
%

\section{Proving the norm relations (Theorem \ref{norm relation})}
\label{proof of norm relation}

Theeorem \ref{norm relation} is directly analogous to the norm-compatibility relations for Euler systems constructed from Siegel units; specifically, it is the analogue in our context of \cite[Theorem 3.3.2]{LLZ14}. Exactly as in \emph{op.cit.}, it is simplest not to prove the theorem directly but rather to deduce it from a related result concerning cohomology classes on the symmetric spaces $Y_F^*(m, m\n)$, analogous to Theorem 3.3.1 of \emph{op.cit.}. Note that these symmetric spaces are not connected for $m > 2$, but have $\phi(m)$ connected components; this will allow us to give a tidy conceptual interpretation of the sum over $j \in (\Z / p^r \Z)^*$ appearing in the definition of $\subc\Phi_{\n, a}^r$.

\subsection{Rephrasing using the spaces $Y_F^*(m, m\n)$}

\begin{proposition}
 For any $a \in \roi_F$, the element $\matrd{1}{a}{0}{1}$ normalises $U^*(m, m\n) \subset \GLt^*(\A_F^f)$.
\end{proposition}

\begin{proof} Easy check. \end{proof}

We can therefore regard right-translation by $\smallmatrd{1}{a}{0}{1}$ as an automorphism of $Y_F^*(m, m\n)$, and we can consider the composite map
\[ \iota_{m, \n, a}: Y_{\Q}^*(m, mN) \hookrightarrow Y_F^*(m, m\n)
\labelrightarrow{\smallmatrd{1}{-a}{0}{1}} Y_F^*(m, m\n),\]
where the first arrow is injective (as soon as $m\n$ is divisible by some integer $\ge 4$) by the same argument as in Proposition \ref{prop:goodinclusion}. Note also that the components of $Y_F^*(m, m\n)$ are indexed by $(\Z / m\Z)^*$, with the fibre over $j$ corresponding to the component containing the image of $\smallmatrd j {}{}1 \in \GLt^*(\A_F^f)$; and the action of $\smallmatrd{1}{a}{0}{1}$ preserves each component.

\begin{remark-numbered}
The change of sign appears because we are comparing left and right actions.
\end{remark-numbered}

\begin{definition}
 We define $\subc\mathcal{Z}_{m, \n, a}$ to be the image of $\subc C_{mN} \in \h^1(Y_{\Q}(m, mN), \Z)$ under pushforward via $\iota_{m, \n, a}$, and $\subc\mathcal{Z}_{m, \n, a}(j)$ the projection of $\subc\mathcal{Z}_{m, \n, a}$ to the direct summand of $\h^1(Y_F^*(m, m\n), \Z)$ given by the $j$-th component, so that
 \[ \subc\mathcal{Z}_{m, \n, a} = \sum_j \subc\mathcal{Z}_{m, \n, a}(j).\]
\end{definition}

Exactly as in the situation of Beilinson--Flach elements, these $\mathcal{Z}$ elements turn out to be closely related to the $\Phi$'s defined above (compare \cite[Proposition 2.7.4]{LLZ14}). We consider the map
\[ s_m:  Y_F^*(m, m\n) \to Y_{F, 1}^*(\n) \]
given by the action of $\matrd{m}{0}{0}{1}$ (corresponding to $(z, t) \mapsto (z/m, t/m)$ on $\uhs$).

\begin{proposition}
 \label{Z-and-Xi}
 We have $(s_m)_*\left( \subc\mathcal{Z}_{m, \n, a}(j) \right) = \subc \Xi_{m, \n, j a}$, and hence
 \[ \subc \Phi_{\n, a}^r = \sum_{j} (s_{p^r})_*\left( \subc\mathcal{Z}_{p^r, \n, a} (j) \right) \otimes [j].\]
\end{proposition}

Before proceeding to the proof, we note the following lemma:

\begin{lemma}
 The pushforward of $\subc C_{m^2 N}$ along the map
 \[ Y_{\Q, 1}(m^2 N) \to Y_\Q(1(m), mN), \]
 given by $z \mapsto mz$ on $\uhp$, is $\subc C_{m N}$.
\end{lemma}
\begin{proof} This follows from the well-known norm-compatibility relations of the Siegel units, cf.\cite[Lemma 2.12]{Kat04}.
\end{proof}

\begin{proof}[Proof of Proposition \ref{Z-and-Xi}]
 For each $j \in (\Z / m\Z)^\times$, we have a diagram
 \[
\xymatrix@R=10mm@C=10mm{
 Y_F^*(m, mN)^{(1)} \ar[r]^{\smallmatrd{1}{-ja}{0}{1}}   \ar[d]^{\smallmatrd{j}{0}{0}{1}} &
 Y_F^*(m, m\n)^{(1)} \ar[d]^{\smallmatrd{j}{0}{0}{1}}\\
Y_F^*(m, mN)^{(j)} \ar[r]^{\smallmatrd{1}{-a}{0}{1}} & Y_F^*(m, m\n)^{(j)}.
}
 \]
 In other words, if we identify $Y_F^*(m, mN)^{(j)}$ with $\Gamma_F^*(m, mN) \backslash \uhs$ via $\smallmatrd{j}{0}{0}{1}$, the restriction to this component of the right action of $\smallmatrd{1}{-a}{0}{1}$ on the adelic symmetric space corresponds to the left action of $\smallmatrd{1}{ja}{0}{1}$ on $\uhs$.

 With these identifications, we see that the map
 \[ \kappa_{ja/m} : Y^*_{F, 1}(m^2 \n) \to Y^*_{F, 1}(\n) \]
 used in the definition of $\Xi_{m, \n, ja}$ factors as
 \begin{align*} Y^*_{F, 1}(m^2 \n) &\labelrightarrow{\smallmatrd{1}{0}{0}{m}} Y^*_{F}(1(m), m\n)\\
  &\labelrightarrow{\cong} Y^*_{F}(m, m\n)^{(j)}\\
 &\labelrightarrow{\smallmatrd{1}{-a}{0}{1}} Y^*_{F}(m, m\n)^{(j)}\\
 &\labelrightarrow{s_m} Y^*_{F, 1}(\n).
 \end{align*}
 Pushforward along the first map is compatible with pushforward along the corresponding map on $\uhp$, which sends $\subc C_{m^2N}$ to $\subc C_{mN}$ by the previous lemma.
\end{proof}

\begin{corollary}
 The classes $\subc \Xi_{m, \n, a}$ and $\subc \mathcal{Z}_{m, \n, a}$ depend only on the image of $a$ in the quotient $\roi_F / (m \roi_F + \Z)$.
\end{corollary}

\begin{proof}
 If $b \in \Z$, the action of $\smallmatrd{1}{b}{0}{1}$ on $Y_{\Q}(m, mN)$ fixes the cohomology class ${}_c C_{mN}$, as this class is the pullback of a class on $Y_{\Q, 1}(mN)$. Since the actions of $\smallmatrd{1}{b}{0}{1}$ on $Y_\Q(m, mN)$ and $Y_{F}^*(m, m\n)$ are compatible, we see that $\subc \mathcal{Z}_{m, \n, a} = \subc \mathcal{Z}_{m, \n, a + b}$ for any $a \in \roi_F$ and $b \in \Z$, as required. The corresponding result for $\subc \Xi_{m, \n, a}$ now follows from the previous proposition.
\end{proof}

\subsection{A norm relation for zeta elements}

In this section, we formulate and prove a norm relation for the zeta elements $\subc\mathcal{Z}_{m, \n, a}$ which is analogous to Theorem \ref{norm relation}, but simpler to prove.

\begin{definition}
 For $p$ prime, define a map
 \[\tau_{p}: Y_F^*(p m, p m\n) \longrightarrow Y_F^*(m,m\n)\]
 by composing the right-translation action of $\smallmatrd{p}{0}{0}{1} \in \GLt^*(\A^f_F)$ with the natural projection.
\end{definition}

\begin{theorem}\label{norm relation zeta}
 Suppose $p$ is a prime with $p \mid m$, and suppose that $a \in \roi_F$ maps to a generator of the quotient $\roi_F / (p \roi_F + \Z) \cong \Z / p \Z$. Then we have the norm relation
 \[(\tau_{p})_* (\subc \mathcal{Z}_{p m, \n, a}) = (U_p)_*(\mathcal{Z}_{m, \n, a}).\]
\end{theorem}

For simplicity, we give the proof under the slightly stronger hypothesis that $p \mid \n$ (rather than just that every prime above $p$ divides $\n$, which is our running hypothesis). This only makes a difference if $p$ is ramified in $F$, and the proof can be extended to handle this extra case at the cost of slightly more complicated notation; we leave the necessary modifications to the interested reader.

Firstly, note that there is a commutative diagram
\begin{equation}
\label{tau commutative diagram}
\xymatrix{
Y_F^*(p m, p m \n) \ar[r]^{\text{\tiny{pr$_1$}}} \ar[rd]_{\tau_p} &
Y_F^*(p m, m\n)  \ar[r]^{\text{\tiny{pr$_2$}}} &
 Y_F^*(m (p), m\n)\ar[ld]_{\pi_2} \ar[d]_{\pi_1}\\
& Y_F^*(m, m\n) &Y_F^*(m, m\n)
},
\end{equation}
where the top maps are the natural projection maps, $\tau_p$ is the twisted degeneracy map of the previous section, and $\pi_1, \pi_2$ are the degeneracy maps of Section \ref{hecke operators}.

\begin{lemma}\label{first norm relation}
 Let $\n' = (p)^{-1} \n.$ Under pushforward by the natural projection map
 \[\mathrm{pr}_1: Y_F^*(pm, pm\n) \longrightarrow Y_F^*(pm, m\n) = Y_F^*(pm,pm\n'),\]
 we have
 \[(\mathrm{pr}_1)_* (\subc \mathcal{Z}_{p^{r+1},\n,a}) = \subc \mathcal{Z}_{p^{r+1},\n',a}.\]
\end{lemma}
\begin{proof}
 This is immediate from the corresponding norm-compatibility property of the Siegel units, which is Proposition \ref{siegel unit properties}. Compare \cite[Theorem 3.1.1]{LLZ14}.
\end{proof}

So we need to compare the classes $\mathcal{Z}_{pm,\n',a}$ and $\mathcal{Z}_{m,\n,a}$. Note that these both involve the same Siegel unit ${}_c g_{0, 1/mN}$. Let us write $u_a$ for the element $\smallmatrd 1{-a} 0 1$.

\begin{definition}
 \begin{itemize}
  \item[(i)]Let $\alpha_{pm,\n,a}$ denote the composition of the maps
  \[Y_\Q(pm, mN) \hookrightarrow Y_F^*(pm, m\n) \labelrightarrow{u_a} Y_F^*(pm, m\n) \labelrightarrow{\mathrm{pr}_2} Y_F^*(m(p),m\n).\]
  \item[(ii)]Let $\iota_{m,\n,a}$ denote, as above, the composition of the maps
  \[Y_\Q(m,m N) \hookrightarrow Y_F^*(m,m\n) \labelrightarrow{u_a} Y_F^*(m,m\n).\]
 \end{itemize}
\end{definition}

The following lemma is the key component in the proof of Theorem \ref{norm relation zeta}.

\begin{lemma}\label{cartesian}
 Suppose that $a \in \roi_F$ is a generator of $\roi_F/(p\roi_F + \Z)$. Then:
 \begin{itemize}
  \item[(i)] The map $\alpha_{pm,\n,a}$ is injective.
  \item[(ii)] The diagram
  \[
  \xymatrix@C=20mm{
  Y_{\Q}(pm,mN) \ar@{^{(}->}[r]^{\alpha_{pm,\n,a}}\ar[d] &
   Y_F^*(m(p),m\n) \ar[d]^{\pi_1}\\
  Y_{\Q}(m,mN) \ar@{^{(}->}[r]^{\iota_{m,\n,a}}& Y_F^*(m,m\n)
}
  \]
  is Cartesian, where the left vertical arrow is the natural projection.
 \end{itemize}
\end{lemma}
The proof of this lemma is taken essentially verbatim from \cite[Lemma 7.3.1]{LLZ16}, where the analogous result is proved for real quadratic fields.

\begin{proof}
 To prove part (i), note that the image of $\alpha_{pm,\n,a}$ is the modular curve of level
 \[\GLt(\A_{\Q}^f)\cap u_a^{-1}U_F(m(p),m\n)u_a.\]
 This intersection is the set of $\smallmatrd{r}{s}{t}{u} \in \GLt(\widehat{\Z})$ such that
 \[\matrd{r-at}{s+a(r-u)-a^2t}{t}{at+u} \equiv I \mod \matrd{m}{pm}{m\n}{m\n}.\]
 We want to show that this is equal to $U_\Q(pm,mN)$. Clearly, any $\smallmatrd{r}{s}{t}{u}$ in the intersection satisfies
 \[\matrd{r}{s}{t}{u} \equiv \matrd{1}{*}{0}{1} \mod \matrd{m}{*}{mN}{mN},\]
 whilst
 \[s + a(r-1) \equiv 0 \newmod{mp}.\]
 Suppose $m = p^q m'$ with $m'$ coprime to $p$. We know that both summands are zero modulo $m$, so it suffices to check that they are both zero modulo $p^{q+1}$. Since $a$ generates $\roi_F/(p\roi_F + \Z)$, $\{1, a\}$ is a basis of $(\roi_F / p^{q+1}\roi_F) \otimes \Zp$ as a module over $\Z / p^{q+1}\Z$; so both summands must be individually zero modulo $p^{q+1}$. But this means precisely that $\smallmatrd{r}{s}{t}{u} \in U_{\Q}(mp, mN)$, as required.

 Part (ii) follows from the observations that the horizontal maps are both injections and that both vertical maps are finite of degree $p^2$.
\end{proof}

\begin{remark-numbered}Since this lemma is crucial to the proof, we expand slightly on what part (i) really says. The map $\alpha_{mp,\n,a}$ is $p$-to-$1$ on connected components, in the sense that the preimage of a single component of $Y_F^*(p^r(p),p^r\n)$ contains $p$ connected components of $Y_F^*(p^{r+1},p^rN)$. The condition on $a$ ensures that the map $u_a$ `twists' these $p$ components away from each other inside that single component of the target space, so that their images are disjoint. In particular, the result would certainly fail without this condition; for instance, if $a = 0$ then the map factors through $Y_\Q(m(p), mN)$.
\end{remark-numbered}

\begin{proof}[Proof of Theorem \ref{norm relation zeta}]
 The Cartesian diagram of Lemma \ref{cartesian} shows that
 \[(\alpha_{mp,\n,a})_*(\subc C_{mN}) = (\pi_1)^*(\subc \mathcal{Z}_{m,\n,a}).\]
 But by definition,
 \begin{align}\label{zeta equality}
 (\alpha_{mp,\n,a})_*(\subc C_{mN}) &= (\mathrm{pr}_2)_*(\subc \mathcal{Z}_{mp,\n',a})\notag \\
 &= (\mathrm{pr}_2)_*(\mathrm{pr}_1)_*(\subc \mathcal{Z}_{mp,\n,a}),
 \end{align}
 where the second equality follows from Lemma \ref{first norm relation}. From the commutative diagram \eqref{tau commutative diagram}, we know that $(\tau_p)_* = (\pi_2)_*(\mathrm{pr}_2)_*(\mathrm{pr}_1)_*$, whilst by definition $(U_p)_* = (\pi_2)_*(\pi_1)^*.$ Hence applying $(\pi_2)_*$ to equation \eqref{zeta equality} gives the result.
\end{proof}

We can now deduce the proof of the main theorem:

\begin{proof}[Proof of Theorem \ref{norm relation}]
 We need to show that, for each $j \in (\Z/p^r \Z)^\times$, we have
 \[ \sum_{\substack{k \in \Z / p^{r + 1} \Z \\ k = j \bmod p^r}} \subc \Xi_{p^{r+1}, \n, ka} = (U_p)_* \cdot \subc \Xi_{p^r, \n, ja}. \tag{\dag}\]

 We have a commutative diagram
 \[
  \xymatrix@C=20mm{
  		 \bigsqcup_k Y_F^*(p^{r+1}, p^{r+1}\n)^{(k)} \ar[r]^-{\tau_p}\ar[d]^{s_{p^{r+1}}}  &
  	  Y_F^*(p^r, p^r\n)^{(j)} \ar[d]^{s_{p^r}}\\
Y_{F, 1}^*(\n) \ar[r] &
Y_{F, 1}^*(\n).
}
 \]
 The left-hand side of $(\dag)$ is exactly the pushforward of $\sum_k \subc \mathcal{Z}_{p^{r+1}, \n, a}(k)$ along the left vertical arrow (where again we are using the notation $x(k)$ for the projection of $x$ to the $k$-th component). Theorem \ref{norm relation zeta} shows that the pushforward of the same element along $\tau_p$ is $(U_p)_* \subc \mathcal{Z}_{p^{r}, \n, a}(j)$. So it suffices to check that the operators $(U_p)_*$ on $Y_F^*(p^r, p^r \n)^{(j)}$ and on $Y_{F, 1}^*(\n)$ are compatible under $s_{p^r}$, which is clear by inspecting a set of single-coset representatives (using our running assumption that all primes above $p$ divide $\n$).\\
\\
The case where $r=0$ is a special case, since we must exclude the term $\subc \Xi_{p^1,\n,0}$ from the above sum, introducing the $-1$ term of the theorem.
\end{proof}

\section{Asai--Eisenstein elements in higher weights}\label{higher weights}

In the previous sections, we have defined compatible systems of classes in the Betti cohomology of the spaces $Y_{F, 1}^*(\n)$ with trivial coefficients. We now extend this to coefficients arising from non-trivial algebraic representations.

We fix, for the duration of this section, a prime $p$ and a finite extension $L$ of $\Qp$ large enough that $F$ embeds into $L$ (and we fix such an embedding). We let $R$ be the ring of integers of $L$. We also choose an ideal $\n$ of $\roi_F$ divisible by all primes above $p$. For convenience, we also assume that $\n$ is divisible by some integer $q \ge 4$; note that this is automatic if $p$ is unramified in $F$ and $p \ge 5$.

\subsection{Coefficients and moment maps}

 As above, we let $V_k(R) = \operatorname{Sym}^k R^2$ be the left $R[\GLt(\Z)]$-module of symmetric polynomials in 2 variables with coefficients in $R$. We will be interested in the dual $T_k(R) = V_k(R)^*$ (the module of symmetric tensors of degree $k$ over $R^2$). We view this as a local system of $R$-modules on $Y_{\Q, 1}(N)$, for any $N \ge 4$, in the usual way.

 Similarly, we have $R[\GLt(\roi_F)]$-modules $V_{kk}(R) = \operatorname{Sym}^k R^2 \otimes (\operatorname{Sym}^{k} R^2)^\sigma$, where $\GLt(\roi_F)$ acts on the first factor via the given embedding $\roi_F \hookrightarrow R$ and on the second via its Galois conjugate. We let $T_{kk}(R)$ be the $R$-dual of $V_{kk}(R)$. These give local systems on $Y_{F}^*(U)$ and $Y_F(U)$ for sufficiently small levels $U$.

 The linear functional dual to the second basis vector of $R^2$ defines a $\Gamma^*_{F, 1}(\n p^r)$-invariant linear functional on $\operatorname{Sym}^k (R/p^r)^2$, or on $\left(\operatorname{Sym}^k (R/p^r)^2\right)^\sigma$, and hence an invariant vector in $T_{kk}(R/p^r))$. This can be seen as a section of the corresponding local system, defining a class
 \[
 e_{F, k, r} \in \h^0\left( Y_{F, 1}^*(\n p^r), T_{kk}(R / p^r)\right).
 \]
 Since $R$ is $p$-adically complete, cup-product with this section defines a ``moment'' map
 \[
  \operatorname{mom}^{kk}: \varprojlim_t \h^\bullet(Y_{F, 1}^*(\n p^t), \Z) \otimes R \to \h^\bullet(Y_{F, 1}^*(\n), T_{kk}(R)).
 \]
 This is the Betti cohomology analogue of the moment maps in \'etale cohomology of modular curves considered in \cite[\S 4]{KLZ17}.

By Lemma \ref{lemma:level-compat}, the family of classes $\left( \Phi_{\n p^t, a}^{k, r}\right)_{t \ge 0}$ is compatible under pushforward, so it is a valid input to the maps $\operatorname{mom}^{kk}$ (after base-extending from $R$ to the group ring $R[(\Z/p^r)^\times]$).

\begin{definition}
 We let $\subc \Phi_{\n,a}^{k, r} \in \h^2\Big(Y_{F, 1}^*(\n),T_{kk}(R)\Big) \otimes_{R}R[(\Z/p^r)^\times]$ be the image of the compatible system
 \(\left( \Phi_{\n p^t, a}^{k, r}\right)_{t \ge 0}\)
 under $\operatorname{mom}^{kk}$.
\end{definition}

The action of the Hecke operator $(U_p)_*$ is well-defined both on $\h^2\big(Y_{F, 1}^*(\n),T_{kk}(R)\big)$ and on the inverse limit $\varprojlim_t \h^2\big(Y_{F, 1}^*(\n p^t), \Zp\big)$, and the maps $\operatorname{mom}^{kk}$ commute with this operator (cf.~\cite[Remark 4.5.3]{KLZ17}). So we deduce immediately from Theorem \ref{norm relation} that the classes $\subc \Phi_{\n,a}^{k, r}$, for any fixed $k \ge 1$ and varying $r$, satisfy the same norm-compatibility relation as the $k = 0$ classes.

\subsection{Relation to the weight $2k$ Eisenstein class}

We will later relate the $\subc \Phi_{\n,a}^{k, r}$ to values of $L$-functions. For this purpose the definition above, via a $p$-adic limiting process, is inconvenient; so we now give an alternative description of the same classes via higher-weight Eisenstein series for $\GLt / \Q$, directly generalising the classes obtained in weight 2 from realisations of Siegel units.

Let $k \ge 0$. The local system $T_{k}(\C)$ is exactly the flat sections of a vector bundle $T_{k, \mathrm{dR}}$ with respect to a connection $\nabla$ (the Gauss--Manin connection). The vector bundle $T_{k, \mathrm{dR}}$ is algebraic over $\Q$, and there is a comparison isomorphism
\begin{equation}
\label{eq:dRcomparison}
 \h^1(Y_{\Q, 1}(N), T_{k}(\Q)) \otimes \C \cong \h^1_{\mathrm{dR}}\left(Y_{\Q, 1}(N), T_{k, \mathrm{dR}}\right)\otimes_{\Q} \C.
\end{equation}
Moreover, the pullback of $T_{k, \mathrm{dR}}$ to the upper half-plane is the $k$-th symmetric tensor power of the relative de Rham cohomology of $\C / (\Z\tau + \Z)$, so it has a canonical section $(\mathrm{d}w)^{\otimes k}$, where $w$ is a coordinate on $\C$.

\begin{proposition}
 There exists a class $\operatorname{Eis}^k_N \in \h^1(Y_{\Q, 1}(N), T_k(\Q))$ whose image under the comparison isomorphism \eqref{eq:dRcomparison} is the class of the differential form
 \[
  -N^k F^{(k+2)}_{1/N}(\tau)\, \mathrm{d}w^{\otimes k} \,\mathrm{d}\tau.
 \]
\end{proposition}

\begin{proof}
 By work of Beilinson--Levin \cite{beilinsonlevin94}, there exists a \emph{motivic Eisenstein class}
 \[
  \operatorname{Eis}^k_{\mathrm{mot}, N} \in \h^1_{\mathrm{mot}}\big(Y_{\Q, 1}(N), T_{k, \mathrm{mot}}(k+1)\big),
 \]
 where $T_{k, \mathrm{mot}}$ is a relative Chow motive over $Y_1(N)$ with coefficients in $\Q$ (cut out by a suitable idempotent inside the relative motive of the $k$-fold fibre product of the universal elliptic curve over $Y_1(N)$).

 This motivic cohomology group admits a realisation map $r_{\mathrm{dR}}$ to $\h^1_{\mathrm{dR}}(Y_{\Q, 1}(N), T_{k, \mathrm{dR}})$, and $\operatorname{Eis}^k_{\mathrm{dR}, N} \coloneqq r_{\mathrm{dR}}(\operatorname{Eis}^k_{\mathrm{mot}, N})$ is given by the class of the differential form
 \[
   -N^k (2\pi i)^{k + 1}F^{(k+2)}_{1/N}(\tau)\, \mathrm{d}w^{\otimes k} \,\mathrm{d}\tau.
 \]
 (This is a restatement of a result of \cite{beilinsonlevin94}, but we refer to \cite[Theorem 4.3.3]{KLZ1} for the statement in this particular form). There is also a Betti realisation map
 \[ r_B: \h^1_{\mathrm{mot}}\big(Y_{\Q, 1}(N), T_{k, \mathrm{mot}}(k+1)\big) \to \h^1(Y_{\Q, 1}(N), T_k(\Q)(k+1)),
 \]
 where the twist $(k+1)$ in Betti cohomology denotes tensor product with $\Q(k+1) = (2 \pi i)^{k+1} \Q \subset \C$; and these two realisations are compatible under the comparison isomorphism \eqref{eq:dRcomparison} (see \cite[\S 2.2]{KLZ1}). Identifying $\Q(k+1)$ with $\Q$ in the obvious manner, we obtain a class $\operatorname{Eis}^k_{N} \in \h^1(Y_{\Q, 1}(N), T_k(\Q))$ whose image under the comparison isomorphism \eqref{eq:dRcomparison} is $(2\pi i)^{-(k+1)}\operatorname{Eis}^k_{\mathrm{dR}, N}$, as required.
\end{proof}

Via base-extension, we can consider $\operatorname{Eis}^k_N$ as a class in $\h^1(Y_{\Q, 1}(N), T_k(\Qp))$. This class does not generally lie in the lattice $\h^1(Y_{\Q, 1}(N), T_k(\Zp))$; but for any $c > 1$ as above, there exists a class $\subc\operatorname{Eis}^k_N \in \h^1(Y_{\Q, 1}(N), T_{k}(\Zp))$ such that the equality
\[
 {}_c \operatorname{Eis}^k_N = \left(c^2 - c^{-k} \langle c \rangle\right) \operatorname{Eis}^k_N
\]
holds in $\h^1(Y_{\Q, 1}(N), T_{k}(\Qp))$. (This follows from Kings' theory of $p$-adic interpolation of Eisenstein classes \cite{kings15}; see \cite[Theorem 4.4.4]{KLZ17}.)

Letting $R$ be as in the previous section, for any $j \in \{0, \dots, k\}$ we can regard $T_{2k-2j}(R)$ as a $\SLt(\Z)$-invariant submodule of the $\SLt(\roi_F)$-module $T_{kk}(R)$, via the \emph{Clebsch--Gordan map}
\begin{equation}
 \label{eq:CG}
 \operatorname{CG}^{[k, k, j]}: T_{2k-2j}(R) \to T_{kk}(R),
\end{equation}
(normalised as in \cite[\S 5.1]{KLZ1}). Thus we obtain a map
\[
 (\iota_{m, \n, a})_* \circ \operatorname{CG}^{[k, k, j]}: \h^1\big(Y_{\Q}(m, mN), T_{2k-2j}(\Zp)\big) \otimes_{\Zp} R \to \h^2\big(Y_{F}^*(m, m\n), T_{kk}(R)\big).
\]

\begin{definition}
 Let $\subc \Xi_{m, \n, a}^{k, j} \in \h^2\big(Y_{F, 1}^*(\n), T_{kk}(R)\big)$ be the image of $(\iota_{m, \n, a})_* \operatorname{CG}^{[k, k, j]}\left(\subc\operatorname{Eis}^{2k-2j}_{m N}\right)$ under restriction to the identity component followed by $(s_m)_*$. We similarly write $\Xi_{m, \n, a}^{k, j}$ (without $c$) for the analogous element with $L$-coefficients, defined using $\operatorname{Eis}^k_{m N}$.
\end{definition}

This definition is convenient for $p$-adic interpolation, but to relate this element to special values it is convenient to have an alternative description involving pushforward along the map $\kappa_{a/m}: Y_{F, 1}^*(m^2\n) \to Y_{F, 1}^*(\n)$, as above. (Note that if $p \mid m$, this pushforward map only acts on cohomology with coefficients in $T_{kk}(L)$, not $T_{kk}(R)$, since it corresponds to the action of a matrix whose entries are not $p$-adically integral.)

\begin{lemma}
 \label{Z-and-Xi2}
 As elements of $\h^2\big(Y_{F, 1}^*(\n), T_{kk}(L)\big)$ we have
 \[ \Xi_{m, \n, a}^{k, j} = m^{j} \cdot (\kappa_{a/m})_*\left( \iota_*\mathrm{CG}^{[k,k,j]}(\operatorname{Eis}^{2k-2j}_{m^2 N})\right).\]
\end{lemma}

\begin{proof}
 This follows from the definition of $\subc \Xi_{m, \n, a}^{k, j}$ above in exactly the same way as Proposition \ref{Z-and-Xi} (which is the case $j=k=0$), noting that the Clebsch--Gordan maps at levels $Y^*_{F}(m, m\n)$ and $Y_{F, 1}^*(m^2 \n)$ differ by the factor $m^j$; compare the proof of \cite[Theorem 5.4.1]{KLZ17}.
\end{proof}

\begin{proposition}
 For any $r \ge 0$ we have
 \begin{align*}
  {}_c \Phi^{k, r}_{\n, a} &=
 \sum_{t \in (\Z/p^r)^\times} {}_c \Xi_{p^r, \n, at}^{k, 0} \otimes [t] \\
 \\
 &= \left(c^2 - c^{-2k} [c]^2 \langle c \rangle\right)\cdot \sum_{t \in (\Z/p^r)^\times} \Xi_{p^r, \n, at}^{k, 0} \otimes [t],
\end{align*}
 where the first equality takes place in $\h^2(Y_{F, 1}^*(\n), T_{kk}(R)) \otimes_R R[(\Z / p^r)^\times]$ and the second after base-extension to $L$.
\end{proposition}

\begin{proof}
 This proposition is very close to \cite[Proposition 5.1.2]{KLZ17} so we only briefly sketch the proof. There is a $\GLt / \Q$ moment map $\mom^{k}$ for any $k \ge 0$, and one sees easily that the maps $\mom^k$ and $\mom^{kk}$ are compatible via the inclusion $Y_{\Q, 1}(N) \hookrightarrow Y^*_{F, 1}(\n)$. However, the main theorem of \cite{kings15} shows that the higher-weight Eisenstein classes are exactly the moments of the family of Siegel-unit classes $\subc C_{Np^\infty}$, up to a factor depending on $c$.
\end{proof}

There is an analogous statement for $j \ne 0$, but this can only be formulated after reduction modulo $p^r$:

\begin{proposition}\label{prop:0-to-j}
 For $r \ge 1$, as classes in $\h^2(Y_{F, 1}^*(\n), T_{kk}(R/p^r))$ we have
 \[ {}_c \Xi_{p^r, \n, a}^{k, j} = (a - a^\sigma)^j j!\binom{k}{j}^2 {}_c \Xi_{p^r, \n, a}^{k, 0}.\]
\end{proposition}

\begin{proof}
 The proof of this proposition is identical to the corresponding statement for \'etale cohomology of Hilbert modular varieties, which is Corollary 8.1.5 of \cite{LLZ16}.
\end{proof}

%
%
\section{The $p$-adic Asai $L$-function}
\label{p-adic L-function}
In this short section, we put together the norm-compatibility and $p$-adic interpolation relations proved above in order to define a measure on $\Zp^\times$ with values in a suitable eigenspace of the Betti $\h^2$. This will be our $p$-adic $L$-function.

To ease the notation, we will assume for the rest of the paper that $p$ is odd. Similar arguments -- with some additional care -- should also hold for $p=2$, but we leave this case to the interested reader.


\subsection{Constructing the measure}

As in \S5, let $L$ be a finite extension of $\Qp$, with a chosen embedding of $F$ into $L$, and write $R$ for the ring of integers in $L$. In previous sections, we defined the elements
\[
\subc \Phi_{\n,a}^{k, r} = \sum_{t\in(\Z/p^r)^\times} \subc \Xi^{k, 0}_{p^r,\n,at}\otimes[t] \in \h^2\big(Y_{F, 1}^*(\n),T_{kk}(R)\big)\otimes R[(\Z/p^r)^\times],
\]
for $k \ge 0$ and $r \ge 0$. We also showed that if $a$ is a generator of $\roi_F/(p\roi_F + \Z)$, then under the natural projection maps in the second factor, we have
\[
\pi_{r+1}(\subc \Phi_{\n,a}^{r+1}) =(U_p)_*\left(\subc \Phi_{\n,a}^r\right) \quad \text{for $r \ge 1$}.
\]

\begin{definition}
Let us write
\begin{align*}
\mathcal{L}_k(\n, R) &\defeq \h^2(Y_{F, 1}^*(\n),T_{kk}(R)) / (\text{torsion}),\\
 \text{and}\quad \mathcal{L}_k^{\mathrm{ord}}(\n, R) &\defeq e_{\mathrm{ord},*}\mathcal{L}_k(\n, R),
 \end{align*}
where $e_{\mathrm{ord},*} \defeq \lim_{n\rightarrow\infty}(U_p)_*^{n!}$ is the ordinary projector.
\end{definition}

Clearly $e_{\mathrm{ord},*}\h^2(Y_{F, 1}^*(\n),T_{kk}(R))$ is an $R$-direct-summand of $\h^2(Y_{F, 1}^*(\n),T_{kk}(R))$, which is a finitely-generated $R$-module, since $Y_{F, 1}^*(\n)$ is homotopy-equivalent to a finite simplicial complex. On this direct summand, $(U_p)_*$ is invertible, so we may make the following definition:

\begin{definition}
 Define
 \[
 \subc \Phi_{\n, a}^{k, \infty} \defeq \big[(U_p)_*^{-r} e_{\mathrm{ord},*}(\subc \Phi_{\n,a}^{k, r})\big]_{r\ge 1} \in \mathcal{L}_k^{\mathrm{ord}}(\n, R) \otimes_R R[[\Zp^\times]],
 \]
 where $R[[\Zp^\times]] = \varprojlim_r R[(\Z / p^r)^\times]$ is the Iwasawa algebra of $\Zp^\times$ with $R$-coefficients.
\end{definition}

We can interpret $R[[\Zp^\times]]$ as the dual space of the space of continuous $R$-valued functions on $\Zp^\times$. For $\mu \in R[[\Zp^\times]]$ and $f$ a continuous function, we write this pairing as $(\mu, f) \mapsto \int_{\Zp^\times} f(x) \mathrm{d}\mu(x)$.

\begin{proposition}
 \label{prop:interpPhi}
 For $j$ an integer with $0 \le j \le k$, and $\chi: \Zp^\times \to \Cp^\times$ a finite-order character of conductor $p^r$ with $r \ge 1$, we have
 \[
  \int_{\Zp^\times} x^j \chi(x)\, \mathrm{d}\,\subc \Phi_{\n, a}^{k, \infty}(x) = \frac{1}{(a - a^\sigma)^j j!\binom{k}{j}^2} (U_p)_*^{-r} e_{\mathrm{ord},*} \sum_{t \in (\Z/p^r)^\times}\chi(t)  \subc \Xi_{p^r,\n, at}^{k, j}
 \]
 as elements of $L(\chi) \otimes_R \mathcal{L}_k^{\mathrm{ord}}(\n, R)$. For $\chi$ trivial we have
 \[
  \int_{\Zp^\times} x^j \, \mathrm{d}\,\subc \Phi_{\n, a}^{k, \infty}(x) = \frac{1}{(a - a^\sigma)^j j!\binom{k}{j}^2}(1 - p^j(U_p)_*^{-1}) e_{\mathrm{ord},*}  \subc \Xi_{1, \n, a}^{k, j}.
 \]
\end{proposition}

\begin{proof}
 For $j = 0$ this is immediate from the definition of $\subc \Phi_{\n, a}^{k, \infty}$ (with the Euler factor in the case of trivial $\chi$ arising from the fact that the norm of $(U_p)_*^{-1}\, \subc \Phi_{\n, a}^{k, 1}$ is not $\subc \Phi_{\n, a}^{k, 0}$ but $(1 - (U_p)_*^{-1})\, \subc\Phi_{\n, a}^{k, 0}$, by the base case of Theorem \ref{norm relation}).

 The case $j \ge 1$ is more involved. It suffices to show the equality modulo $p^h$ for arbitrarily large $h$. Modulo $p^h$ with $h \ge r$, we have
 \begin{align*}
   (a - a^\sigma)^j j!\binom{k}{j}^2 \int_{\Zp^\times} x^j \chi(x)\, \mathrm{d}\,\subc \Phi_{\n, a}^{k, \infty}(x) & \\
   =(a - a^\sigma)^j j!\binom{k}{j}^2 (U_p)_*^{-h} e_{\mathrm{ord},*} \sum_{t \in (\Z / p^h)^\times} t^j \chi(t) \subc \Xi_{p^h, \n, at}^{k, 0} & \qquad\text{(definition of $\subc \Phi_{\n, a}^{k, \infty}$)}\\
   = (U_p)_*^{-h} e_{\mathrm{ord},*} \sum_{t \in (\Z / p^h)^\times}\chi(t) \subc \Xi_{p^h\n, at}^{k, j} & \qquad\text{(Proposition \ref{prop:0-to-j})} \\
   = (U_p)_*^{-h}e_{\mathrm{ord},*} \sum_{t \in (\Z / p^r)^\times} \chi(t) \Bigg( \sum_{\substack{s \in (\Z/p^h)^\times \\ s = t \bmod p^r}} \subc \Xi_{p^h, \n, as}^{k, j}\Bigg).
 \end{align*}
 The bracketed term is $(U_p)_*^{h-r} \subc \Xi_{p^r, \n, at}^{k, j}$ if $r \ge 1$, while for $r = 0$ it is $(U_p)_*^{h}(1 - p^j (U_p)_*^{-1}) \subc \Xi_{1, \n, a}^{k, j}$, by the same argument as the proof of Theorem \ref{norm relation}.
\end{proof}

 Now suppose $\Psi$ is a Bianchi modular eigenform of parallel weight $(k, k)$ and level $U_{F,1}(\n)$. Recall that if $E$ is the extension of $F$ generated by the Hecke eigenvalues of $\Psi$, and $\mathscr{P}$ a prime of $E$ above $p$, we defined in \S \ref{sect:modsymb} an element
 \[
  \phi_\Psi^* = \jmath^*\left(\omega_{\Psi}/\Omega_{\Psi}\right) \in \h^1_{\mathrm{c}}( Y_{F, 1}^*(\n), V_{kk}(E)),
 \]
 well-defined up to elements of $E^\times$ that are units at $\mathscr{P}$. Enlarging $L$ if necessary, we fix an embedding $E_{\mathscr{P}} \hookrightarrow L$, and regard $\phi^*_\Psi$ as an element of $\h^1_{\mathrm{c}}(Y_{F, 1}^*(\n), V_{kk}(R))$, well-defined modulo $R^\times$.

\begin{assumption}
 We shall assume that the Bianchi modular eigenform $\Psi$ is \emph{ordinary} with respect to this embedding, i.e.~that the $(U_p)^*$-eigenvalue of $\Psi$ lies in $R^\times$.
\end{assumption}

Since the adjoint of $(U_p)_*$ is $(U_p)^*$, this assumption implies that the linear functional on $\mathcal{L}_k(\n, R)$ given by pairing with $\phi^*_\Psi$ factors through projection to the $(U_p)_*$-ordinary part $\mathcal{L}_k^{\mathrm{ord}}(\n, R)$.

We also need to fix a value of $a$, which must generate the quotient $\frac{\roi_F \otimes \Zp}{\Zp}$. It suffices to take $a = \tfrac{1 + \sqrt{-D}}{2}$ if $D = -1 \bmod 4$, and $a = \tfrac{\sqrt{-D}}{2}$ if $D = 0 \bmod 4$; then we have $\roi_F = \Z + \Z a$, and $a - a^\sigma = \sqrt{-D}$.

\begin{definition}\label{def:p-adic l-function}
Define the \emph{$p$-adic Asai $L$-function} $\subc L_p^{\mathrm{As}}(\Psi) \in R[[\Zp^\times]]$ to be
\[\subc L_p^{\mathrm{As}}(\Psi) \defeq \big\langle \phi^*_\Psi, \subc \Phi_{\n,a}^{k, \infty}\big\rangle, \]
where $\langle-, -\rangle$ denotes the (perfect) Poincar\'e duality pairing
\begin{equation}\label{eqn:pairing}
 \h^1_{\mathrm{c}}(Y_{F, 1}^*(\n), V_{kk}(R)) \times
 \frac{\h^2(Y_{F, 1}^*(\n), T_{kk}(R))}{(\mathrm{torsion})} \longrightarrow R.
\end{equation}
\end{definition}

\begin{remark-numbered}
 If we relax the assumption that $\Psi$ be ordinary, and let $h = v_p(c(p\roi_F, \Psi))$ (where the valuation is normalised such that $v_p(p) = 1$), then we can still make sense of $\subc L_p^{\mathrm{As}}(\Psi)$ as long as $h < 1$; however, it is no longer a measure, but a distribution of order $h$. This can be extended to $h < 1 + k$ using the same techniques as in \cite{LZ16}. However, if $k = 0$ and $h \ge 1$ (as in the case of an eigenform associated to an elliptic curve supersingular at the primes above $p$) then we are stuck.
\end{remark-numbered}

\begin{proposition}
 The class $\subc L_p^{\mathrm{As}}(\Psi)$ is invariant under translation by $[-1] \in \Zp^\times$.
\end{proposition}

\begin{proof} This follows from Proposition \ref{prop:xi-sign}, since $\smallmatrd{-1}{}{}{1}_*$ acts trivially on $\omega_{\Psi}$ (and thus on $\phi^*_{\Psi}$).
\end{proof}

If we interpret $R[[\Zp^\times]]$ as the algebra of $R$-valued rigid-analytic functions on the ``weight space'' $\mathcal{W} = \Hom(\Zp^\times, \Cp^\times)$ parametrising characters of $\Zp^\times$, then this proposition shows that $\subc L_p^{\mathrm{As}}(\Psi)$ vanishes identically on the subspace $\mathcal{W}^- \subset \mathcal{W}$ parametrising odd characters.

We close this section by giving notation that will be useful when stating the interpolation properties of $L_p^{\mathrm{As}}(\Psi)$.

\begin{notation}
 Let $\chi$ be a Dirichlet character of conductor $p^r$ for some $r\ge 0$, and let $j$ be any integer. We write
 \[
  \subc L_p^{\mathrm{As}}(\Psi,\chi,j) \defeq \int_{\Zp^\times}\chi(x)x^j\, \mathrm{d}\, \subc L_p^{\mathrm{As}}(\Psi)(x).
 \]
\end{notation}

 \subsection{Getting rid of $c$}\label{sec:getting rid of c}

 \begin{proposition}
  Suppose that the nebentypus character $\varepsilon_\Psi: (\roi_F / \n)^\times \to R^\times$ of $\Psi$ has non-trivial restriction to $(\Z / N\Z)^\times$, and moreover this restriction does not have $p$-power conductor. Then there exists a measure $L_p^{\mathrm{As}}(\Psi) \in L \otimes_R R[[\Zp^\times]]$ such that
  \[ {}_c L_p^{\mathrm{As}}(\Psi) = (c^2 - c^{-2k} \varepsilon_\Psi(c) [c]^2) L_p^{\mathrm{As}}(\Psi)\]
  for all valid integers $c$.
 \end{proposition}

 \begin{proof}
  Some bookkeeping starting from \eqref{eq:cdsymmetry} shows that if $c, d$ are two integers $> 1$, both coprime to $6Np$, then the element
  \[ (d^2 - d^{-2k}[d]^2 \varepsilon_{\Psi}(d)) \cdot {}_c L_p^{\mathrm{As}}(\Psi) \]
  is symmetric in $c$ and $d$. Moreover, since $\varepsilon_{\Psi}$ does not have $p$-power conductor, we can choose $d$ such that $(d^2 - d^{-2k}[d]^2 \varepsilon_{\Psi}(d))$ is a unit in $L \otimes_R R[[\Zp^\times]]$. So if we define
  \[ L_p^{\mathrm{As}}(\Psi) = (d^2 - d^{-2k} \varepsilon_\Psi(d) [d]^2)^{-1} {}_d L_p^{\mathrm{As}}(\Psi), \]
  then this is independent of the choice of $d$ and it has the required properties.
 \end{proof}

 If the restriction $\varepsilon_{\Psi, \Q}$ of $\varepsilon_\Psi$ has $p$-power conductor, then the quotient $L_p^{\mathrm{As}}(\Psi)$ is well-defined as an element of the fraction ring of $R[[\Zp^\times]]$, i.e.~as a meromorphic function on $\mathcal{W}$ with coefficients in $L$. (We shall refer to such elements as \emph{pseudo-measures}.) The only points of $\mathcal{W}$ at which $L_p^{\mathrm{As}}(\Psi)$ may have poles are those corresponding to characters of the form $z \mapsto z^{k+1} \nu(z)$, where $\nu^2 = \varepsilon_{\Psi, \Q}^{-1}$.

 \begin{remark-numbered}
  Note that if $p = 1 \bmod 4$ and $\varepsilon_{\Psi, \Q}(\rho) = (-1)^k$, where $\rho$ is either of the square roots of $-1$ in $\Zp$, then both of the characters at which $L_p^{\mathrm{As}}$ could have a pole actually lie in $\mathcal{W}^-$, so we see immediately that $L_p^{\mathrm{As}}$ is a measure.

  In the remaining cases, where one or both potential poles are in $\mathcal{W}^+$, we suspect that these potential poles are genuine poles if and only if the corresponding complex-analytic Asai $L$-functions have poles (which can only occur if $\Psi$ is either of CM type, or a twist of a base-change form). However, we have not proved this.
 \end{remark-numbered}

%
%
\section{Interpolation of critical $L$-values}
 \label{interpolation}

 In this section, we want to show that the values of the Asai $p$-adic $L$-function at suitable locally-algebraic characters are equal to special values of the complex $L$-function.

 \subsection{Automorphic forms for $G^*$}

  We shall need to work with automorphic forms for the group $G^*$ of Definition \ref{def:G*} above. We refer the reader to \cite{LLZ16} for an account of automorphic forms for the group $G^*$, and their relation to those for $G$, in the analogous setting where $F$ is a \emph{real} quadratic field.

  For $U^* \subset G^*(\hat\Z)$, we let $S_{kk}(U^*)$ denote the space of automorphic forms for $G^*$ of level $U^*$ and weight $(k, k)$. These are defined in the same way as for $G$; that is, they are functions
  \[ G^*(\Q)_+ \backslash G^*(\A_\Q) / U^* \to V_{2k+2}(\C)\]
  transforming appropriately under $\R_{>0} \cdot \SUt(\C)$, and with suitable harmonicity and growth conditions. If $U^* = U \cap G^*$ for an open compact subgroup $U$ of $G(\A_\Q^f)$, then there is a natural pullback map $\jmath^*: S_{kk}(U) \to S_{kk}(U^*)$.

  Any $\f \in S_{kk}(U^*_{F, 1}(\n))$ is uniquely determined by its restriction to $G^*(\R)$, since $Y^*_{F, 1}(\n)$ is connected. This restriction can be described by a Fourier--Whittaker expansion of the form
  \[
   \f\left(\matrd{y_\infty}{x_\infty}{0}{1}\right) = |y_\infty| \sum_{\zeta \in F^\times} W_f(\zeta, \f) W_{\infty}(\zeta y_\infty) e_F(\zeta x),
  \]
  where $W_f(-, \f)$ is a function on $F^\times$ (supported in $\mathcal{D}^{-1}$). Of course, if $\f = \jmath^*(\Psi)$ for some $\Psi \in S_{kk}(U_{F, 1}(\n))$, then $W_f(-, \f)$ is simply the restriction of $W_f(-, \Psi)$ to $F^\times \subset (\A_F^f)^\times$.

  \begin{remark-numbered}
   The theory of automorphic representations of $G^*$ is more complicated than that of $G$: not all cuspidal representations are globally generic, and the naive formulation of strong multiplicity one is false, due to the presence of non-trivial global $L$-packets. In practice, this means that although cusp forms for $G^*$ do have Fourier--Whittaker expansions, one cannot necessarily recover all of their Fourier--Whittaker coefficients from the action of the Hecke algebra of $G^*(\A_\Q^f)$. However, this will not concern us here, since we will only consider automorphic forms for $G^*$ which are restrictions of eigenforms for $G$, or twists of these.
  \end{remark-numbered}

  \begin{lemma}
   Let $\f \in S_{kk}(U_{F, 1}^*(\n))$, $\chi$ a Dirichlet character of conductor $m$, and $a \in \roi_F / m \roi_F$. Then the function
   \[ R_{a, \chi} \f = \sum_{t \in (\Z/m)^\times} \chi(t) \kappa_{at/m}^*(\f) \]
   is in $S_{kk}(U_{F, 1}^*(m^2 \n))$, and its Fourier--Whittaker coefficients for $\zeta \in \mathcal{D}^{-1}$ are given by
   \[
    W_f(\zeta, R_{a, \chi} \f) = G(\chi)\, \bar\chi(\operatorname{tr}_{F/\Q} a\zeta)\, W_f(\zeta, \f),
   \]
   where $G(\chi) \defeq \sum_{t \in (\Z/m)^\times} \chi(t) e^{2\pi i t/m}$ is the Gauss sum of $\chi$.
  \end{lemma}

  \begin{proof}
   It is clear that $R_{a, \chi}\f$ has level $U_{F, 1}^*(m^2 \n)$, since each term in the sum is invariant under $U_{F, 1}^*(m^2 \n)$. It remains to compute its Fourier--Whittaker coefficients. We have
   \begin{align*}
    W_f(\zeta, R_{a, \chi} \f) &= W_f(\zeta, \f) \sum_{t \in (\Z/m)^\times}  \chi(t) e_F(\zeta a t / m) \\
    &= W_f(\zeta, \f) \sum_{t \in (\Z/m)^\times}  \chi(t) e^{2\pi i t \operatorname{tr}(a\zeta) / m}.
   \end{align*}
   This is 0 unless the integer $\operatorname{tr}(a\zeta)$ is a unit modulo $m$, in which case it is $\chi(\operatorname{tr}(a\zeta))^{-1} G(\chi)$, as required.
  \end{proof}

 \subsection{An integral formula for the Asai $L$-function}

  In this section, we describe an integral formula for the Asai $L$-function of a Bianchi eigenform twisted by a Dirichlet character $\chi$. This is a generalisation of the work of Ghate in \cite{Gha99} (who considers the case where $\chi$ is trivial), and we shall prove our theorem by reduction to his setting using the twisting maps $R_{a, \chi}$.

  Let $0 \le j \le k$, and define
  \[
  I_{\Psi,b,m}^j \defeq \Big\langle \phi_{\Psi}^*, \hspace{4pt}(\kappa_{b/m})_*\iota_* \mathrm{CG}^{[k,k,j]}_* F^{(2k-2j+2)}_{1/m^2N}(\tau)\, \mathrm{d}w^{\otimes 2k-2j} \,\mathrm{d}\tau\Big\rangle,
  \]
  where $\langle -,-\rangle$ denotes the pairing of equation \eqref{eqn:pairing}, $\phi_{\Psi}^* = \jmath^*\phi_{\Psi}$ as before, and we view the Eisenstein class as an element of the Betti cohomology (with complex coefficients) using the standard comparison isomorphism.

  \begin{theorem}
   \label{thm:int formula}
   Let $\chi$ be a Dirichlet character of odd conductor $m$, and let $0 \le j \le k$. Let $a \in \roi_F$ be the value we chose in the remarks before Definition \ref{def:p-adic l-function} (so that $a - a^\sigma = \sqrt{-D}$). Then
   \[
    \sum_{t\in(\Z/m\Z)^\times}\chi(t)I_{\Psi,at,m}^j =
    \begin{cases}
     \frac{C'(k,j)G(\chi)}{(m^2N)^{2k-2j} \Omega_{\Psi}} L^{\mathrm{As}}(\Psi,\chibar,j+1)& \text{if}\ (-1)^j\chi(-1) = 1,\\
   0 & \text{if}\ (-1)^j\chi(-1) = -1,\end{cases}
   \]
   where
   \begin{align*}
   C'(k,j) = (-1)^{k+1}\frac{\sqrt{-D}^{j+1} (j!)^2\binomc{k}{j}^2}{2\cdot (2\pi i)^{j+1} N^{2k-2j}}.
   \end{align*}
  \end{theorem}

  We begin by explaining how to reduce the theorem to the case $m = 1$. Note that the definition of the Asai $L$-function depends only on the pullback $\jmath^*\Psi$, and in fact makes sense for any $\f \in S_{kk}(U_{F, 1}^*(\n))$, whether or not it is in the image of $\jmath^*$, as long as it is an eigenvector for the operators $\langle x \rangle$ for $x \in (\Z / N\Z)^\times$. If these operators act on $\f$ via the character $\varepsilon_\f$, then we can define
  \[
   L^{\mathrm{As}}(\f, s) \defeq L^{(N)}(\varepsilon_{\f}, 2s-2k-2) \sum_{n \ge 1} W_f\left(n/\sqrt{-D}, \f\right) n^{-s}.
  \]
  One sees easily that if $\chi$ is a Dirichlet character of odd conductor, and $a$ is the value we chose above (so that $a - a^\sigma = \sqrt{-D}$), then
  \[ L^{\mathrm{As}}(R_{a,\chi} \jmath^* \Psi, s) = G(\chi) \cdot L^{\mathrm{As}}(\Psi, \chibar, s).\]

  \begin{proposition}
   \label{prop:int formula untwisted}
   Let $\f \in S_{kk}(U_{F, 1}^*(\n))$, and let $N = \n \cap \Z$. Then we have
   \begin{multline*}
     \Big\langle \omega_{\f}, \iota_* \mathrm{CG}^{[k,k,j]}_* F^{(2k-2j+2)}_{1/N}(\tau)\, \mathrm{d}w^{\otimes 2k-2j} \,\mathrm{d}\tau\Big\rangle \\=
   \begin{cases}
   \frac{C'(k,j)}{N^{2k-2j}} L^{\mathrm{As}}(\f,j+1)& \text{if}\ \smallmatrd{-1}{}{}{1}^* \f = (-1)^j \f,\\
   0 & \text{if}\ \ \smallmatrd{-1}{}{}{1}^* \f = (-1)^{j+1} \f.\end{cases}
   \end{multline*}
  \end{proposition}

  The proof of the proposition is very similar to the work of Ghate \cite{Gha99}, but our conventions are a little different, so we shall give a proof using our conventions in an appendix; see Corollary \ref{cor:int formula}.

  Applying this proposition to $R_{a,\chi} \jmath^* \Psi$ and dividing by $\Omega_\Psi$ proves Theorem \ref{thm:int formula}, since $\smallmatrd{-1}{}{}{1}^*$ acts on $R_{a,\chi} \jmath^* \Psi$ as $\chi(-1)$.

\subsection{Interpolation of critical values}\label{sec:interpolation}

We now use the integral formula of Theorem \ref{thm:int formula} to relate the values of the measure $L_p^{\mathrm{As}}(\Psi)$ to critical values of the classical Asai $L$-function.

\begin{theorem}\label{thm:interpolation}
Let $p$ be an odd prime. Let $\Psi$ be an ordinary Bianchi eigenform of weight $(k,k)$ and level $U_{F,1}(\n)$, where all primes above $p$ divide $\n$, with $U_p$-eigenvalue $\lambda_p = c(p\roi_F, \Psi)$. Let $\chi$ be a Dirichlet character of conductor $p^r$, and let $0 \le j \le k$.
\begin{itemize}
\item[(a)] If $\chi(-1)(-1)^j = 1$, then

\[L_p^{\mathrm{As}}(\Psi, \chi, j) =
 \frac{C(k,j) \mathcal{E}_p(\Psi, \chi, j) G(\chi)}{\Omega_{\Psi}} \cdot  L^{\mathrm{As}}(\Psi,\chibar,j+1),\]
where
\[C(k,j) \defeq (-1)^{k+1}\frac{j!\cdot \sqrt{-D}}{2\cdot (2\pi i)^{j+1}},\quad
\mathcal{E}_p(\Psi, \chi, j) \defeq
\begin{cases}
\left( 1 - \tfrac{p^j}{\lambda_p}\right) & \text{if $r = 0$,}\\
\left( p^{j} \lambda_p^{-1}\right)^r & \text{if $r > 0$.}
\end{cases}
 \]
\item[(b)]
If $\chi(-1)(-1)^j = -1$, then
\[ L_p^{\mathrm{As}}(\Psi,\chi,j) = 0.\]
\end{itemize}
\end{theorem}
\begin{remark}
Up to rescaling $\Omega_{\Psi}$, this is precisely the interpolation formula predicted by Coates--Perrin-Riou (see \S\ref{sec:coates-perrin-riou}).
\end{remark}
\begin{proof}
For convenience, let $e[r]$ denote the operator $(U_p^{-r})_*e_{\mathrm{ord},*}$ if $r \ge 1$, and $(1 - p^j (U_p^{-1})_*)e_{\mathrm{ord},*}$ if $j = 0$. By the definition of the measure and Proposition \ref{prop:interpPhi}, we have
\[
 L_p^{\mathrm{As}}(\Psi,\chi,j) = \frac{1}{\sqrt{-D}^j j!\binomc{k}{j}^2}\sum_{t\in(\Z/p^{r})^\times} \chi(t)\bigg\langle \phi_{\Psi}^*, e[r] \Xi^{k,j}_{p^{r},\n,at}\bigg\rangle.
\]

We know that $(U_p)^*$ is the adjoint of $(U_p)_*$, and $\phi_{\Psi}^*$ is a $(U_p)^*$ eigenvector with unit eigenvalue $\lambda_p$; thus the adjoint of $e[r]$ acts on $\phi_\Psi^*$ as $p^{-jr} \mathcal{E}_p(\Psi, \chi, j)$, so we have
\[
L_p^{\mathrm{As}}(\Psi,\chi,j) = \frac{\mathcal{E}_p(\Psi, \chi, j)}{p^{jr} \sqrt{-D}^j j!\binomc{k}{j}^2}\sum_{t\in(\Z/p^{r})^\times} \chi(t)\bigg\langle \phi_{\Psi}^*, \Xi^{k,j}_{p^{r},\n,at}\bigg\rangle.
\]
Now, by Lemma \ref{Z-and-Xi2}, we have $\Xi^{k,j}_{p^{r},\n,at} = p^{jr} (\kappa_{at/p^r})_*\iota_* \mathrm{CG}^{[k,k,j]}_*\left(\operatorname{Eis}^{2k-2j}_{p^{2r}N}\right)$, and hence
\[
\big\langle \phi_{\Psi}^*, \Xi_{p^r,\n,at}\big\rangle =p^{jr} \Big\langle (\mathrm{CG}^{[k,k,j]})^*\iota^*\kappa_{at/p^r}^*(\phi_{\Psi}^*), \operatorname{Eis}^{2k-2j}_{p^{2r}N}\Big\rangle,
\]
where the first cup product is at the level of $\Gamma_{F,1}^*(\n)\backslash\uhs$, and the second cup product is at the level of $\Gamma_1(p^{2r}N)\backslash\uhp.$ Now work at the level of complex coefficients. We know that
\begin{align*}
 \mathrm{Eis}^{2k-2j}_{p^{2r}N} &= \mathrm{Eis}^{2k-2j}_{p^{2r}N}\\
& =- (p^{2r}N)^{2k-2j} F^{(2k-2j+2)}_{1/p^{2r}N}(\tau)\, \mathrm{d}w^{\otimes 2k-2j} \,\mathrm{d}\tau,\end{align*}
Accordingly, we see that
\[
 L_p^{\mathrm{As}}(\Psi,\chi,j) =
-\frac{ (p^{2r}N)^{2k-2j}\mathcal{E}_p(\Psi, \chi, j)}
{\sqrt{-D}^j j!\binomc{k}{j}^2} \sum_{t\in (\Z/p^r)^\times} \chi(t) I^j_{\Psi,at,p^r},
\]
where $I^j_{\Psi,b,m}$ is as defined in the previous section. Using Theorem \ref{thm:int formula}, we see that this expression vanishes unless $\chi(-1)(-1)^j = 1$, in which case we have
\begin{multline*}
L_p^{\mathrm{As}}(\Psi,\chi,j) =
-\frac{ \mathcal{E}_p(\Psi, \chi, j)}
{\Omega_{\Psi}\sqrt{-D}^j j!\binomc{k}{j}^2} \times C'(k,j)G(\chi)L^{\mathrm{As}}(\Psi,\chibar,j+1)\\
= \frac{C(k,j) \mathcal{E}_p(\Psi, \chi, j) G(\chi)}{\Omega_{\Psi}} \cdot  L^{\mathrm{As}}(\Psi,\chibar,j+1),
\end{multline*}
which completes the proof of the theorem.
\end{proof}

As an immediate corollary, we get an identical interpolation formula for ${}_c L_p^\mathrm{As}(\Psi)$ with the additional factor $(c^2 - c^{2j-2k} \varepsilon_{\Psi}(c)\chi(c)^2)$.

\begin{remark-numbered} \
 \begin{enumerate}
  \item[(i)] The factor $\mathcal{E}_p(\Psi, \chi, j)$ is non-zero if $r \ge 1$. If $r = 0$ then this factor vanishes if and only if $k = 0$, $\Psi$ is new at the primes above $p$, and $\varepsilon_\Psi(p) = 1$. In this case the $p$-adic $L$-function has an exceptional zero at the trivial character. For exceptional zeroes of the \emph{standard} $p$-adic $L$-function of a Bianchi cusp form, a theory of $\mathcal{L}$-invariants was developed in \cite{BW17}; it would be interesting to investigate analogues of this for the Asai $L$-function.

  \item[(ii)] The measure $L_p^{\mathrm{As}}(\Psi)$ depends on the choice of $\sqrt{-D}$ fixed at the start; indeed, this choice was used to pick a value of $a \in \roi_F$, which in turn was used to construct the Asai--Eisenstein elements. This choice is further encoded by the appearance of $\sqrt{D} = i\sqrt{-D}$ in the interpolation formula. Choosing the other square root simply scales the measure by $-1$. The measure also depends on the choice of period $\Omega_{\Psi}$, and again a different choice changes the measure up to a scalar.

  \item[(iii)] If the Bianchi eigenform $\Psi$ (or, more precisely, the automorphic representation it generates) is the base-change lift of an elliptic modular eigenform $f$ of weight $k+2$ and character $\varepsilon_f$, then the complex Asai $L$-function factors as
  \[ L^{\mathrm{As}}(\Psi, \chi, s) = L(\operatorname{Sym}^2 f, \chi, s) L(\chi \varepsilon_f \varepsilon_F, s - k - 1),\]
  where $\varepsilon_K$ is the quadratic character associated to $K$. Note that all three $L$-functions in the above formula have critical values at integer points $s = 1 + j$ with $0 \le j \le k$ and $(-1)^j \chi(-1) = 1$. By a comparison of interpolating properties at these points, one can verify that if $f$ is ordinary at $p$, then there is a corresponding factorisation of $L_p^\mathrm{As}(\Psi)$ as a product of a shifted $p$-adic Dirichlet $L$-function and Schmidt's $p$-adic $L$-function for $\operatorname{Sym}^2 f$.

  This factorisation shows, in particular, that the possibility of poles of the $p$-adic Asai $L$-function is a genuine aspect of the situation, rather than a shortcoming of our method: if $\varepsilon_f \varepsilon_F = 1$, then one of these factors is the $p$-adic Riemann zeta function $\zeta_p(s-k-1)$, which has a simple pole at $s = k + 2$. If $f$ has CM by an imaginary quadratic field $K$ (with $K \ne F$, so that $\Psi = \operatorname{BC}(f)$ is cuspidal), then there is a second abelian factor $L(\chi \varepsilon_f \varepsilon_K, s-k-1)$; this gives rise to examples where both of the zeros of the factor $c^2 - c^{-2k}\varepsilon_\Psi(c)[c]^2$ correspond to genuine poles of $L_p^{\mathrm{As}}(\Psi)$.
 \end{enumerate}
\end{remark-numbered}

 \appendix

 \section{Appendix: an integral formula for the Asai $L$-function}

 In this appendix, we provide a slight extension of Ghate's results on the rationality of Asai $L$-values. The main result is given in Theorem \ref{thm:integral formula}. Most of the arguments go through almost exactly as in \cite{Gha99}. However, in order to present the $p$-adic interpolation computations in the main body of the paper in the most natural way, we have adopted conventions which differ from those of \emph{op.cit.} in a few places: firstly, the Eisenstein series which we work with are slightly different from Ghate's, as is our normalisation for the Clebsch--Gordan map; secondly, we focus on computing the special values of the Asai $L$-series in the left half of the critical strip (i.e.~$1 \le s \le k+1$) rather than the right half (i.e.~$k+2 \le s \le 2k+2$) considered by Ghate. Hence we shall explain how to modify his results for our setting.

 \begin{notation}
  Let $\f \in S_{kk}(U_{F, 1}^*(\n))$ be a Bianchi modular form of weight $(k, k)$ and level $U_{F,1}^*(\n)$, and suppose that $\f$ is an eigenvector for the action of $\smallmatrd{ -1}{}{} 1$ (necessarily with eigenvalue either $+1$ or $-1$).

  \begin{itemize}
   \item[(i)] Write $\psi_{\f}$ for the Dirichlet character given by restricting the character of $\f$ from $(\roi_F / \n)^\times$ to $(\Z/N)^\times$.
   \item[(ii)] Let $\omega_{\f}$ be the differential form on $\uhs$ associated to $\f$ as in \S\ref{sect:modsymb}. This descends to a differential form on $\Gamma_{F,1}^*(\n)\backslash\uhs$.
   \item[(iii)]
     Write $\f_n$, $0 \le n \le 2k+2$, for the $n$-th coordinate projection of $\f$ in the standard basis $\{X^{2k+2-n} Y^{n}\}$ of $V_{2k+2}(\C)$, which is a $\C$-valued function on $\uhs$.
  \end{itemize}
 \end{notation}

\subsection{Definition of the integrals}

 Pulling back $\omega_{\f}$ via the inclusion $\iota : Y_1(N) =\Gamma_1(N) \backslash \uhp \hookrightarrow \Gamma_{F,1}^*(\n)\backslash \uhs$, we obtain a $\Gamma_{\Q, 1}$-invariant differential on $\uhp$.

 For each $j$ with $0 \le j \le k$, we have a map of $\GLt^+(\R)$-modules
 \[
  \operatorname{CG}^{[k, k, j], \star}: V_k(\C) \otimes V_k(\C) \to V_{2k-2j}(\C)
 \]
 given by the transpose of the map $\operatorname{CG}^{[k, k, j]}$ of \eqref{eq:CG} above. Let us write $(\iota^*\omega_{\f})^{2k-2j}$ for $\mathrm{CG}^{[k,k,j]*}(\iota^*\omega_{\f}) \in \h^1_{\mathrm{cusp}}\big(Y_{\Q,1}(N),V_{2k-2j}(\C)\big)$.

 \begin{definition}
  \label{def:I}
  Define
  \[I_{\f}^j(s) \defeq \int_{Y_1(N)} (\iota^*\omega_{\f})^{2k-2j}\wedge E^{(2k-2j+2)}_{1/N}(\tau, s)\, \mathrm{d}w^{\otimes 2k-2j} \,(\mathrm{d}x+i\mathrm{d}y).\]
 \end{definition}

 Here $E^{(2k-2j+2)}_{1/N}(\tau, s)$ is as in (\ref{eqn:Eis}), and $\mathrm{d}w^{\otimes 2k-2j}$ is considered as a section of a symmetric tensor power of $T_{1, \mathrm{dR}}$, which we are identifying with the de Rham cohomology of the universal elliptic curve $\C / (\Z \tau + \Z)$.

  The goal of this appendix is to evaluate $I^j_\f(s)$ in terms of  $L^{\mathrm{As}}(\f, s + 2k-j + 2)$. We shall only use this for one specific value of $s$, namely $s = -1-2k + 2j$, as the left-hand side of Proposition \ref{prop:int formula untwisted} is exactly $I_{\f}^j(-1-2k+2j)$ (by the adjointness of pushforward and pullback). However, since this $s$-value lies outside the domain of convergence of the sums defining the Eisenstein series and Asai $L$-function, we need to work initially with an arbitrary $s$ such that $\Re(s) \gg 0$ and obtain the result by analytic continuation.

 \subsection{An explicit description of the differential}

  We recall the description of $(\iota^*\omega_{\f})^{2k-2j}$ obtained in \cite{Gha99}. Ghate uses a slightly different normalisation for the projection $V_k \otimes V_k \to V_{2k-2j}$ (see Lemma 2 of \emph{op.cit.}): he considers the map sending a polynomial $\delta(X, Y, \bar X, \bar Y) \in V_{kk}(\C)$ to
 \[
  \tfrac{1}{(j!)^2}\nabla^j \delta(X, Y) \Big|_{X = \overline{X},Y = \overline{Y}}, \quad\text{where } \nabla \coloneqq \left( \frac{\partial^2}{\partial X\partial \overline{Y}} - \frac{\partial^2}{\partial Y\partial \overline{X}} \right).
  \]

  \begin{proposition}
   We have
   \[ \operatorname{CG}^{[k, k, j], \star}(\delta) = (-1)^j j! \cdot \tfrac{1}{(j!)^2}\nabla^j(\delta) \Big|_{X = \overline{X},Y = \overline{Y}}.\]
  \end{proposition}

  \begin{proof}
   It is clear that these two maps must be proportional (since $V_{2k-2j}$ is irreducible and appears with multiplicity one in $V_{kk}$), so we must compute the constant of proportionality. Let $w^{[s, k-s]}$, for $0 \le s \le k$, be the basis of $T_k(\C)$ dual to the basis $X^s Y^{k-s}$ of $V_k(\C)$. An explicit formula for $\operatorname{CG}^{[k, k, j]}$ in this basis is given in \cite[Proposition 5.1.2]{KLZ1}\footnote{Note that the formula appears with some typographical errors in \emph{op.cit.}: the terms $(k-r + i)!$ and $(k' -r' +j-i)!$ should be $(k-r-i)!$ and $(k'-r'-j+i)!$ respectively.}. One case of this formula is
   \[
   \operatorname{CG}^{[k, k, j]}\left( w^{[0, 2k-2j]}\right) = \\
   \sum_{i = 0}^j (-1)^i \frac{(k - i)!(k + i - j)!}{(k-j)!(k-j)! } w^{[i, k-i]} \otimes w^{[j - i, k -j + i]}.
   \]
   Pairing $\operatorname{CG}^{[k, k, j]}\left( w^{[0, 2k-2j]}\right)$ with $X^0 Y^{k} \bar{X}^j \bar{Y}^{k-j}$ pulls out the $i = 0$ term in the sum, and yields the non-zero scalar $\frac{k!}{(k-j)!}$. Meanwhile, applying $\tfrac{1}{(j!)^2}\nabla^j$ to $X^0 Y^{k} \bar{X}^j \bar{Y}^{k-j}$ gives the polynomial $ \tfrac{(-1)^j k!}{j!(k-j)!} Y^{k-j}\bar{Y}^{k-j}$, and if we set $\bar Y = Y$ and pair this with $w^{[0, 2k-2j]}$, we obtain $\tfrac{(-1)^j k!}{j!(k-j)!}$. So the constant of proportionality is $(-1)^j j!$ as required.
  \end{proof}

  We thus get extra factors of $(-1)^j j!$ compared to \cite{Gha99}, so we have the following formula:

 \begin{proposition}[{\cite[\S5.2]{Gha99}}]\label{Ghate description}
  We have
  \[(\iota^* \omega_{\f})^{2k-2j}(x,y) = (-1)^j j!\sum_{\ell = 0}^{2k-2j}(A_\ell dx +2B_\ell dy)y^{2k-j-\ell}(X-xY)^\ell Y^{2k-2j-\ell},\]
  where
  \begin{align*}
  A_\ell &= \sum_{n = 0}^{k+1}(-1)^ng_{n,j}\left(x,y\right)a(j,\ell,n),\\
  B_\ell &= \sum_{n = 0}^{k+1}(-1)^ng_{n,j}\left(x,y\right)b(j,\ell,n),
  \end{align*}
  \[g_{n,j}(z,t) = \binom{2k+2}{n}^{-1} \left\{\begin{array}{ll} \f_n(z,t) + (-1)^{k+1-n+j}\f_{2k+2-n}(z,t) &: n \in \{0,...,k\}\\
  \f_{k+1}(z,t) &: n = k+1,\end{array}\right. \]
  and $a(j,\ell,n)$ and $b(j,\ell,n)$ are constants given by
  \begin{align*}a(j,\ell,n) = \binomc{k}{j}^2(-1)^{\frac{k+n-\ell-j}{2}}\sum_{t=0}^j(-1)^t\binomc{j}{t}&\bigg[\smallbinomc{k-j}{\frac{k-j-\ell+n}{2}-t}\smallbinomc{k-j}{\frac{3k-3j-\ell-n}{2}+t}\\
  & + \smallbinomc{k-j}{\frac{k-j-\ell+n-2}{2}-t}\smallbinomc{k-j}{\frac{3k-3j-\ell-n+2}{2}+t}\bigg],
  \end{align*}
  \[b(j,\ell,n) = \binomc{k}{j}^2(-1)^{\frac{k+n-1-\ell-j}{2}}\sum_{t=0}^j(-1)^t\binomc{j}{m}\smallbinomc{k-j}{\frac{k-j-\ell+n-1}{2}-t}\smallbinomc{k-j}{\frac{3k-3j-\ell-n+1}{2}+t}.\]
  (Here it is understood that $\smallbinomc{c}{d} = 0$ if $d$ is negative or not an integer).\qed
 \end{proposition}

 \subsection{Pairing with Eisenstein series and integrating}
We shall compute $I^j_\f(s)$ using a Rankin--Selberg unfolding argument. Factoring out the denominator in equation \eqref{eqn:Eis}, we have
\begin{equation}
 \label{eisenstein description}
 E_{1/N}^{(k)}(\tau,s)
 = \frac{\Gamma(s+k)N^{2s+k}}{(-2\pi i)^k\pi^{s}}\sum_{\substack{t\ge 1\\(t,N) = 1}}t^{-k-2s} \langle t \rangle^{-1}
 \left[\sum_{\gamma \in \Gamma_\infty\backslash \Gamma_1(N)}(\mathrm{Im}(\tau)^s \mid_k \gamma)\right],
\end{equation}
where as usual $\langle u \rangle$ denotes the action of any element $\smallmatrd{a}{b}{c}{d} \in \Gamma_0(N)$ such that $d\equiv u\newmod{N}$. The adjoint of $\langle u^{-1}\rangle$ is $\langle u\rangle$, which commutes with $\iota^*$. Hence, considering $\omega_{\f}$ as a differential on $\uhs$, we have
 \[
 \langle u \rangle\iota^*\omega_{\f} =  \psi_{\f}(u)(\iota^* \omega_{\f}).
 \]
 In particular, for $\operatorname{Re}(s) \gg 0$ we have
 \begin{align*}
 I_{\f}^j(s) =  \frac{\Gamma(s+k_j)N^{2s+k_j}}{(-2\pi i)^{k_j}\pi^{s}}\Bigg[\sum_{\substack{u\ge 1\\(u,N)=1}}\psi_{\f}(u)u^{-k_j-2s} \Bigg] \left[\int_{\Gamma_\infty \backslash \uhp} (\iota^*\omega_{\f})^{2k-2j}\wedge \operatorname{Im}(\tau)^s \mathrm{d}w^{\otimes 2k-2j} \mathrm{d}\tau\right],
 \end{align*}
 where we have defined $k_j = 2k-2j+2$ for ease of notation, and $\Gamma_\infty \defeq \{\smallmatrd{1}{n}{0}{1}: n \in \Z\}$. The first bracketed term is simply the Dirichlet $L$-series $L^{(N)}(\psi_{\f}, 2s + k_j)$.

 We shall now substitute the explicit expression for $(\iota^*\omega_{\f})^{2k-2j}$ into the above. The basis vectors $X, Y$ correspond to the natural basis $(\tau, 1)$ of $\h_1(\C / (\Z\tau + \Z), \Z) \cong \Z\tau + \Z$, so we have $\langle X, \mathrm{d}w \rangle = \tau$ and $\langle Y, \mathrm{d}w\rangle = 1$; hence
 \[ \langle (X-xY)^{\ell} Y^{2k-2j-\ell}, \mathrm{d}w^{\otimes 2k-2j}\rangle = \langle (X-xY), \mathrm{d}w\rangle^{\ell} \cdot \langle Y, \mathrm{d}w\rangle^{2k-2j-\ell} = (iy)^\ell.\]
 So we obtain the following formula:

 \begin{proposition}

 	We have
 	\[
 		I_{\f}^j(s) = H(s,k,j) \cdot L^{(N)}(\psi_{\f},2s+k_j)
\sum_{n=0}^{k+1}(-1)^n c'(j,n) \int_0^\infty\!\!\!\int_0^1 g_{n,j}y^{2k-j+s}\mathrm{d}x\, \mathrm{d}y,
	\]
	where
 \begin{equation}\label{eqn:H}
 H(s,k,j) \defeq \frac{(-1)^jj! \Gamma(s+k_j)N^{2s+k_j}}{(-2\pi i)^{k_j}\pi^{s}}
 \end{equation}
 and
	\[ \pushQED{\qed}
 	c'(j,n) \defeq \sum_{\ell=0}^{2k-2j} i^{\ell} \big[i\cdot a(j,\ell,n) - 2b(j,\ell,n)\big].\qedhere
	\]\popQED
 \end{proposition}

To compute this, we look at components separately. In particular, consider the integral
 \[\mathcal{J}_{\f}^{n,r}(s) \defeq \binomc{2k+2}{n}^{-1}\int_0^\infty\int_0^1 \f_{n}\left(x,y\right)y^{r+s}dx dy.\]
 The expression $I_{\f}^j$ is a linear combination of $\mathcal{J}_{\f}^{n,2k-j}$ for varying $n$.
 \begin{proposition}\label{prop:J}
  We have
  \begin{align*}\mathcal{J}_{\f}^{n,r}(s) = G(n,r,s)\sum_{0 \neq m \in \Z}(-\sigma(m))^{k+1-n}W_f\left(\tfrac{m}{\sqrt{-D}}, \f\right)|m|^{-(s+r+2)},
  \end{align*}
  where$\sigma(m)$ is the sign of $m$ and
  \[ G(n,r,s) = \Gamma\left(\frac{s-k+n+r+1}{2}\right)\Gamma\left(\frac{s+k-n+r+3}{2}\right) 2^{s+r}\left(\frac{\sqrt{D}}{4\pi}\right)^{s+r+2}.\]
 \end{proposition}

 \begin{proof}
  Recall the Fourier expansion of $\f_n$ from \cite[eqn.\ (16)]{Gha99}; we have
  \begin{align*}
  \f_n\left(x,y\right) = y\binomc{2k+2}{n}\sum_{\zeta\in F^\times}&W_f(\zeta, \f)\left(\frac{\zeta}{i|\zeta|}\right)^{k+1-n}\\
  &\times K_{n-k-1}(4\pi|\zeta|y)e^{2\pi i\mathrm{Tr}_{F/\Q}(\zeta x)}.
  \end{align*}
  We substitute this in and isolate the terms involving $x$.

  \begin{lemma}
   Suppose $W_f(\zeta, \f) \neq 0$. Then we have
   \[\int_0^1 e^{2\pi i \mathrm{Tr}_{F/\Q}(\zeta x)} dx = \left\{\begin{array}{ll}1 &: \zeta = m/\sqrt{-D} \text{ with $0\neq m \in \Z$},\\
   0 &: \text{otherwise}.\end{array}\right.\]
  \end{lemma}
  \begin{proof}
   We know that $W_f(\zeta, \f) \neq 0$ only if $\zeta \in \Diff^{-1}$. But since $\Diff^{-1}$ is the inverse different of $F/\Q$, for any such $\zeta$ we have $\mathrm{Tr}_{F/\Q}(\zeta) \in \Z$. Hence the integral is 0 unless $\mathrm{Tr}_{F/\Q}(\zeta) = 0$, which happens if and only if $\zeta = m/\sqrt{-D}$ for some non-zero integer $m$, as required.
  \end{proof}
  In particular, the sum collapses to one over non-zero integers, and we remove the terms involving $x$. We also see that if $\zeta = m/\sqrt{-D}$, then $\zeta/i|\zeta| = i|\zeta|/\zeta = -\sigma(m)$. We conclude that
  \begin{equation} 
  \label{eq:Jnrs}
  \begin{aligned}
  \mathcal{J}_{\f}^{n,r}(s) &= \sum_{0 \neq m \in \Z}(-\sigma(m))^{k+1-n}W_f\left(\tfrac{m}{\sqrt{-D}}, \f\right)\int_0^\infty y^{s+r+1}K_{n-k-1}\left(\frac{4\pi|m| y}{\sqrt{D}}\right) dy \\
  &=\left( \sum_{0 \neq m \in \Z}(-\sigma(m))^{k+1-n}W_f\left(\tfrac{m}{\sqrt{-D}}, \f\right)|m|^{-(s+r+2)}\right)\times  G(n,r,s),
  \end{aligned}
  \end{equation}
  where $G(n,r,s) = \int_0^\infty y^{s+r+1}K_{n-k-1}\left(\frac{4\pi y}{\sqrt{D}}\right) dy$. A standard integral formula (see \cite{Hid94}, page 485) gives the explicit expression for $G(n,r,s)$ in the statement of Proposition \ref{prop:J}, completing the proof.
 \end{proof}

 \begin{corollary}\label{cor:dx term vanishes}
  \begin{itemize}
   \item[(i)] If $k+1-n+j$ is odd, then
   \[c'(j,n)\int_0^\infty \int_0^1 g_{n,j}\left(x,y\right)y^{r+s} dxdy = 0,\]
   where $g_{n,j}$ is as defined in Proposition \ref{Ghate description}.

   \item[(ii)] For $\ell$ odd in the sum in Proposition \ref{Ghate description}, only the $a(j,\ell,n)$ terms contribute to $I_{\f}^j(s)$. For $\ell$ even, only the $b(j,\ell,n)$ terms contribute.
  \end{itemize}
 \end{corollary}
 \begin{proof}
  For (i), we assume first that $0 \le n \le k$. Recall that $g_{n,j}(z,t)$ is defined as a scalar multiple of $\f_n(z,t) + (-1)^{k+1-n+j}\f_{2k+2-n}(z,t).$ Proposition \ref{prop:J} shows that
  \[\mathcal{J}_{\f}^{n,r}(s) = \mathcal{J}_{\f}^{2k+2-n,r}(s)\]
  (a consequence of the fact that $K_r = K_{-r}$). Suppose that $k+1-n+j$ is odd; then unwinding the integral of part (i), we find that it is equal to
  \[ \mathcal{J}_{\f}^{n,r}(s)- \mathcal{J}_{\f}^{2k+2-n,r}(s) = 0,\]
  as required.

  This leaves the case $n = k+1$. In this case, $k+1-n+j = j$. But if $j$ is odd, then in the definition of both $a(j,\ell,k+1)$ and $b(j,\ell,n+1)$, the $t$ and $j-t$ terms appear with opposite signs and are easily checked to negate each other. Hence $c'(j,n) = 0$, and we are done.

  For part (ii), for the $n$ term to contribute to $I_{\f}^j(s)$, by (i) $k+j+n$ must be odd. If $\ell$ is odd, it follows that all of the binomial coefficients defining $b(j,\ell,n)$ involve non-integer entries, so vanish; and if $\ell$ is even, similarly the $a(j,\ell,n)$ are all 0.
 \end{proof}
By including these parity conditions, and exploiting the symmetry between binomial coefficients to renormalise the sum to be from $0$ to $2k+2$ (in the process removing the disparity between $0 \le n \le k$ and $n = k+1$) we arrive at the following.
\begin{corollary}\label{cor:parity form}
	We have
	\[
	I_{\f}^j(s) = H(s,k,j) \cdot L^{(N)}(\psi_{\f},2s+k_j)
	\sum_{\substack{n=0\\ k+j+n+1\text{ even}}}^{2k+2}(-1)^{j} 2 c(j,n)\mathcal{J}_{\f}^{n,2k-j}(s),
	\]
	where
	\[
	c(j,n) \defeq \frac{(-1)^{k+1}}{2}\sum_{\substack{\ell=0\\ \ell\text{ even}}}^{2k-2j} i^\ell \big[a(j,\ell-1,n) - 2b(j,\ell,n)\big],
	\]
	(normalised to agree with \cite[p.628]{Gha99}) where we take $a(j,-1,n) = 0$ for all $j,n$.
\end{corollary}

The following is the main result of this appendix.

 \begin{theorem}\label{thm:integral formula}
  Let $0 \le j \le k$. Then
  \[ I_{\f}^j(s) = \begin{cases}
  C(k,j,s) L^{\mathrm{As}}(\f, s+2k-j+2) &\text{if}\ \smallmatrd{-1}{}{}{1}^* \f = (-1)^j \f\\
  0 &\text{if}\ \smallmatrd{-1}{}{}{1}^* \f = (-1)^{j+1} \f,
  \end{cases}
  \]
  where
  \[
  	C(k,j,s) \defeq (-1)^{j}\left(\frac{\sqrt{D}}{2\pi}\right)^{s+2k-j+2} H(s,k,j) G_\infty'(s)
  \]
  for
   	\[
 		 G_\infty'(s) =
\frac{(-1)^j \sqrt{\pi} \binomc{k}{j}^2}{2^s}\frac{\Gamma(\tfrac{s}{2}+k-j+1)}{\Gamma(\tfrac{s+1}{2})} \frac{\Gamma(s+k-j+1)\Gamma(s+2k-j+2)}{\Gamma(s+2k-2j+2)}.
  \]

 Here $I_{\f}^j(s)$ was defined in Definition \ref{def:I} and $H(s,k,j)$ was defined in equation \eqref{eqn:H}.
 \end{theorem}
 \begin{proof}
 	In Proposition \ref{prop:J}, since $k+j+n+1$ is even, the exponent on $(-\sigma(m))$ is equal to $j$. In this expression, we collapse the sum over $m$ to be only over positive integers, using the relation $W_f(-\zeta,\f) = W_f(\zeta,\smallmatrd{-1}{}{}1^* \f)$. If the eigenvalue of $\f$ for this involution is $(-1)^{j+1}$, then the terms for $m$ and $-m$ sum to 0, so the sum over all $m$ vanishes and $I^j_{\f}$ is identically 0, which proves the theorem in this case. If the eigenvalue is $(-1)^j$ then the terms for $-m$ and $m$ are equal; assuming that this is the case, we have
 	  	\[
 	  		\mathcal{J}_{\f}^{n,2k-j}(s) = 2G(n,2k-j,s)\sum_{m \ge 1}W_f\left(\tfrac{m}{\sqrt{-D}}, \f\right) m^{-(s+2k-j+2)}.
 		\]
	Substituting this into Corollary \ref{cor:parity form}, the sum over positive $m$ and the Dirichlet $L$-series combine to give exactly $L^{\mathrm{As}}(\f,s+2k-j+2),$ which we can bring outside the sum over $n$. The remaining terms are
 	\begin{align*}
 		 4 H(s,k,j)
 		\sum_{\substack{n=0\\ k+j+n+1\text{ even}}}^{2k+2}(-1)^{j} c(j,n)G(n,2k-j,s)
 		 = (-1)^{j}\left(\frac{\sqrt{D}}{2\pi}\right)^{s+2k-j+2} H(s,k,j) G_\infty'(s),
 	\end{align*}
 	where
 	\[
 		G_\infty'(s) \defeq \sum_{\substack{n=0\\ k+j+n+1\text{ even}}}^{2k+2} c(j,n)\Gamma\left(\frac{s+k+n-j+1}{2}\right)\Gamma\left(\frac{s+3k-n-j+3}{2}\right)
 	\]
 	is what Ghate calls $G_\infty'(s,f)$ in \cite[p.628]{Gha99}. The main result of \cite{LS14} showed that $G_\infty'(s)$ has the form stated in the theorem.
 \end{proof}

 We have proved Theorem \ref{thm:integral formula} assuming $\Re(s) \gg 0$; but both sides have analytic continuation to all $s \in \C$, so the theorem remains valid on this larger domain. Thus we may evaluate at the value $s = -1 -2k +2j$, and we find:
 \begin{corollary}\label{cor:int formula}
  Let $0 \le j \le k$. Then
  \[
   I_{\f}^j(-1-2k+2j) =
   \begin{cases}
    \frac{C'(k,j)}{N^{2k-2j}}L^{\mathrm{As}}(\f,j+1) & \text{if}\ \smallmatrd{-1}{}{}{1}^* \f = (-1)^j\f, \\\
    0 & \text{if}\ \smallmatrd{-1}{}{}{1}^* \f = (-1)^{1+j}\f,
  \end{cases}
  \]
  where
  \[
  \pushQED{\qed}
  C'(k,j) = (-1)^{k+1}\frac{\sqrt{-D}^{j+1} (j!)^2\binomc{k}{j}^2}{2\cdot (2\pi i)^{j+1}}.
  \qedhere
  \]\popQED
 \end{corollary}

 This proves Proposition \ref{prop:int formula untwisted} of the main text.
%

%
%
\small
\renewcommand{\refname}{\normalsize References}


\setlength{\parskip}{0ex}
\renewcommand{\urladdrname}{ORCID}

\end{document}